\documentclass[10pt]{amsart}
\usepackage{amsmath}
\usepackage{amssymb}
\usepackage{amsfonts}
\usepackage{mathrsfs}
\usepackage{amscd}
\usepackage{graphicx}

\numberwithin{equation}{section}       

\theoremstyle{plain}
\newtheorem{theo}{Theorem}
\newtheorem{prop}{Proposition}[section]

\newtheorem{lemm}[prop]{Lemma}

\newtheorem{theoalph}{Theorem}

\newtheorem{propalph}[theoalph]{Proposition}

\theoremstyle{definition}
\newtheorem{defi}[prop]{Definition}

\theoremstyle{remark}
\newtheorem{rema}[prop]{Remark}

\newtheoremstyle{citing}
  {3pt}
  {3pt}
  {\itshape}
  {}
  {\bfseries}
  {.}
  {.5em}
  {\thmnote{#3}}

\theoremstyle{citing}
\newtheorem*{generic}{}

\newcommand{\partn}[1]{{\smallskip \noindent \textbf{#1.}}}

\newcommand{\C}{\mathbb{C}}
\newcommand{\D}{\mathbb{D}}

\newcommand{\N}{\mathbb{N}}

\newcommand{\R}{\mathbb{R}}

\newcommand{\Z}{\mathbb{Z}}

\newcommand{\cC}{\mathcal{C}}

\newcommand{\cE}{\mathcal{E}}

\newcommand{\cI}{\mathcal{I}}

\newcommand{\cK}{\mathcal{K}}

\newcommand{\cM}{\mathcal{M}}

\newcommand{\cO}{\mathcal{O}}
\newcommand{\cP}{\mathcal{P}}

\newcommand{\cR}{\mathcal{R}}

\newcommand{\cW}{\mathcal{W}}
\newcommand{\cX}{\mathcal{X}}

\newcommand{\fD}{\mathfrak{D}}

\newcommand{\sM}{\mathscr{M}}

\newcommand{\sP}{\mathscr{P}}

\newcommand{\hU}{\widehat{U}}
\newcommand{\hV}{\widehat{V}}
\newcommand{\hW}{\widehat{W}}

\newcommand{\halpha}{\widehat{\alpha}}

\newcommand{\hkappa}{\widehat{\kappa}}

\newcommand{\hrho}{\widehat{\rho}}

\newcommand{\tH}{\widetilde{H}}

\newcommand{\tV}{\widetilde{V}}

\newcommand{\tX}{\widetilde{X}}
\newcommand{\tY}{\widetilde{Y}}

\newcommand{\talpha}{\widetilde{\alpha}}

\newcommand{\tgamma}{\widetilde{\gamma}}

\newcommand{\teta}{\widetilde{\teta}}

\newcommand{\tsigma}{\widetilde{\sigma}}

\newcommand{\tphi}{\widetilde{\phi}}

\renewcommand{\=}{ : = }

\DeclareMathOperator{\diam}{diam}

\DeclareMathOperator{\dist}{dist}

\DeclareMathOperator{\modulus}{mod} 

\newcommand{\wtp}{\widetilde{p}}
\newcommand{\wtm}{\widetilde{m}}

\newcommand{\chicrit}{\chi_{\operatorname{crit}}(c)}
\newcommand{\chicritbase}{\chi_{\operatorname{crit}}(c_0)}
\newcommand{\chiinfC}{\chi_{\operatorname{inf}}^{\C}(c)}
\newcommand{\chiinfR}{\chi_{\operatorname{inf}}^{\R}(c)}
\newcommand{\cl}[1]{{\rm cl}({#1})}
\newcommand{\ene}{\widetilde{n}}

\begin{document}

\title{Low-temperature phase transitions in the quadratic family}
\author{Daniel Coronel}
\address{Daniel Coronel, Facultad de Matem{\'a}ticas, Pontifica Universidad Cat{\'o}lica de Chile, Avenida Vicu{\~n}a Mackenna~4860, Santiago, Chile}
\email{acoronel@mat.puc.cl}
\author{Juan Rivera-Letelier}
\address{Juan Rivera-Letelier, Facultad de Matem{\'a}ticas, Pontifica Universidad Cat{\'o}lica de Chile, Avenida Vicu{\~n}a Mackenna~4860, Santiago, Chile}
\email{riveraletelier@mat.puc.cl}

\begin{abstract}
We give the first example of a quadratic map having a phase transition after the first zero of the geometric pressure function.
This implies that several dimension spectra and large deviation rate functions associated to this map are not (expected to be) real analytic, in contrast to the uniformly hyperbolic case.
The quadratic map we study has a non-recurrent critical point, so it is non-uniformly hyperbolic in a strong sense.
\end{abstract}


\maketitle

%
%

\section{Introduction}
In their pioneer works, Sinai, Bowen and Ruelle~\cite{Sin72, Bow75, Rue76} initiated the thermodynamic formalism of smooth dynamical systems.
They gave a complete description in the case of a uniformly hyperbolic diffeomorphism and a H{\"o}lder continuous potential.
In the last decades there has been a substantial progress in extending the theory beyond this setting.
A complete picture is emerging in real and complex dimension~$1$, see~\cite{BruTod09,IomTod11,MakSmi00,MakSmi03,PesSen08,PrzRiv11,PrzRivinterval} and references therein.
See also~\cite{Sar11,UrbZdu0901,VarVia10} and references therein for (recent) results in higher dimensions.

In this paper we focus in the quadratic family; one of the simplest and yet challenging families of smooth one-dimensional maps.
For a real parameter~$c$ we consider~$2$ dynamical systems arising from the complex quadratic polynomial
$$ f_c(z) \= z^2 + c; $$
the action of~$f_c$ on~$\R$ and the action of~$f_c$ on its complex Julia set.
For each of these dynamical systems and for a varying real number~$t$, we consider the pressure of the geometric potential~$- t \log |Df_c|$.
There are thus~$2$ pressure functions associated to~$f_c$: One in the real setting and another one in the complex setting.
In what follows we use ``geometric pressure function'' to refer to any of these functions; precise definitions and statements are given in~\S\ref{ss:main statements}.

Our main interest are ``phase transitions'' in the statistical mechanics sense:  For a real number~$t_*$ the map~$f_c$ has a \emph{phase transition at~$t = t_*$} if the geometric pressure function is not real analytic at~$t = t_*$.
In the real case, phase transitions might be caused by lack of transitivity, see for example~\cite{Dob09}.
Since these phase transitions are well understood, we restrict our discussion to parameters for which the real map is transitive. 
For~$c = -2$ the map~$f_{-2}$ is a Chebyshev polynomial and it has a phase transition at~$t = - 1$.\footnote{For~$c = -2$ the Julia set of~$f_{-2}$ is the interval~$[-2, 2]$ and both, the real and complex geometric pressure functions of~$f_{-2}$ are given by~$t \mapsto \max \{ - t \log 4, (1 - t) \log 2 \}$.}
The mechanism behind this phase transition, and of any phase transition in the complex setting that occurs at a negative value of~$t$, was explained by Makarov and Smirnov, see~\cite[Theorem~B]{MakSmi00}.\footnote{Makarov and Smirnov showed this type of phase transition is caused by the existence of a gap in the Lyapunov spectrum; more precisely, they showed that if a complex rational map has a phase transition at some~$t < 0$, then there is a finite set of periodic points~$F$ such that there is a definite distance separating the Lyapunov exponents of periodic points in~$F$ and the Laypunov exponents of measures that do not charge~$F$.
Makarov and Smirnov also showed that this type of phase transition is removable in the following sense: The function obtained by omitting the measures that charge~$F$ in the supremum defining the geometric pressure function is real analytic on~$(- \infty, 0)$ and coincides with the geometric pressure function up to the phase transition.
}
Combining the results of Makarov and Smirnov with recent results of Przytycki and the second named author, it follows that for every real parameter~$c \neq - 2$ the map~$f_c$ has at most~$1$ phase transition; moreover, if~$f_c$ has a phase transition, then it occurs at some~$t > 0$.
See~\cite[{\S}A.$3$]{PrzRiv11} for the complex case and~\cite{PrzRivinterval} for the real case.

To describe the possible phase transitions for~$c \neq - 2$, it is useful to distinguish~$3$ complementary cases: $f_c$ uniformly hyperbolic, $f_c$ satisfying the \emph{Collet-Eckmann condition}:
$$ \chicrit \= \liminf_{n \to + \infty} \frac{1}{n} \log |Df_c^n(c)| > 0,\footnote{In the complex setting the Collet-Eckmann condition is usually formulated in such a way that a uniformly hyperbolic map satisfies the Collet-Eckmann condition by vacuity.
Here we use the usual terminology in the real setting, for which a uniformly hyperbolic map does not satisfy the Collet-Eckmann condition.} $$
and the remaining case, when~$f_c$ is not uniformly hyperbolic and does not satisfy the Collet-Eckmann condition.
The Collet-Eckmann condition is one of the strongest and most studied non-uniform hyperbolicity conditions in dimension~$1$, see for example~\cite{AviMor05,BenCar85,GaoShe1111,NowSan98,PrzRivSmi03} and references therein.
Benedicks and Carleson showed that the set of real parameters~$c$ such that~$f_c$ satisfies the Collet-Eckmann condition has positive Lebesgue measure, see~\cite{BenCar85}.
Moreover, Avila and Moreira showed that the set of real parameters~$c$ such that~$f_c$ is not uniformly hyperbolic and does not satisfy the Collet-Eckmann condition has zero Lebesgue measure, see~\cite{AviMor05} and also~\cite{GaoShe1111}.

When~$f_c$ is uniformly hyperbolic, the work of Sinai, Bowen and Ruelle can be adapted to show that the geometric pressure function is real analytic at every real number, see for example~\cite[\S$6.4$]{PrzUrb10}.
That is, if~$f_c$ is uniformly hyperbolic, then it has no phase transitions.

If~$f_c$ is not uniformly hyperbolic and does not satisfy the Collet-Eckmann condition, then the geometric pressure function is non-negative and vanishes for large values of~$t$, see~\cite[Theorem~A]{NowSan98} or~\cite[Corollary~$1.3$]{Riv1204} for the real case and~\cite[Main Theorem]{PrzRivSmi03} for the complex case.
Thus in this case~$f_c$ has a phase transition at the first zero of the geometric pressure function.
Note that this phase transition is associated to the lack of (non-uniform) expansion of~$f_c$.

This paper is focused in the remaining case, when~$f_c$ satisfies the Collet-Eckmann condition.
We show that, contrary to a widespread belief, such a map can have a phase transition at some~$t > 0$.
As a consequence, several dimension spectra and large deviation rate functions associated to such a~$f_c$ are not (expected to be) real analytic, see Remark~\ref{r:non-analytic spectra}.
In the complex setting it also follows that the corresponding integral means spectrum is not real analytic either.

Our construction is very flexible.
We give the simplest example here, of a ``first-order'' phase transition: The geometric pressure function is not differentiable at the phase transition.
In the companion paper~\cite{CorRivb} we modify our construction to obtain an ``high-order'' phase transition: The geometric pressure function is bounded from above and from below by smooth functions that coincide at the phase transition.
To the best of our knowledge it is the first example of a (transitive) smooth dynamical system having such an infinite contact-order phase transition.
Our construction is also robust: In every sufficiently small perturbation of the quadratic family there is a Collet-Eckmann parameter having a phase transition.

The quadratic maps studied here are largely inspired by the conformal Cantor sets with analogous properties studied by Makarov and Smirnov, see~\cite[\S$5$]{MakSmi03}.
There are however several important differences.
Most notably, the conformal Cantor set studied by Makarov and Smirnov is defined through a map having~$2$ affine branches, something that cannot be replicated in a complex polynomial or rational map.

These examples show that lack of (non-uniform) expansion is not the only source of phase transitions.\footnote{In some sense, the phase transitions studied here, as those studied by Makarov and Smirnov, are caused by the irregular behavior of the critical orbit.}
In fact, the quadratic maps studied here satisfy a property that is even stronger than the Collet-Eckmann condition: The critical point is non-recurrent.\footnote{This is usually called the ``Misiurewicz condition'' and it is known to imply the Collet-Eckmann condition, see~\cite{Mis81} for the real setting and~\cite{Man93} for the complex one.}
Thus, no slow recurrence condition, such as the one studied by Benedicks and Carleson~\cite{BenCar85} or by Yoccoz and by Pesin and Senti~\cite{PesSen08}, is sufficient to avoid phase transitions.

\subsection{Statements of results}
\label{ss:main statements}
We consider a set of real parameters close to~$c = - 2$, for which the critical point~$z = 0$ is mapped to a certain uniformly expanding set under forward iteration by~$f_c$, see~\S\ref{s:parameter plane} for details.
For such a parameter~$c$ we have~$f_c(c) > c$, the interval~$I_c \= [c, f_c(c)]$ is invariant by~$f_c$, and~$f_c$ is topologically exact on this set.
We consider both, the interval map~$f_c|_{I_c}$ and the complex quadratic polynomial~$f_c$ acting on its Julia set~$J_c$.

For a real parameter~$c$ denote by~$\sM_c^{\R}$ the space of probability measures supported on~$I_c$ that are invariant by~$f_c$.
For a measure~$\mu$ in~$\sM_c^{\R}$ denote by~$h_\mu(f_c)$ the measure-theoretic entropy of~$f_c$ with respect to~$\mu$ and for each~$t$ in~$\R$ put
$$ P_c^{\R}(t)
\=
\sup \left\{ h_\mu(f_c) - t \int \log |Df_c| \ d\mu \mid \mu \in \sM_c^{\R} \right\}, $$
which is finite.
The function~$P_c^{\R} : \R \to \R$ so defined is called the \emph{geometric pressure function of~$f_c|_{I_c}$}; it is convex and non-increasing.
An invariant probability measure supported on~$I_c$ is an \emph{equilibrium state of~$f_c|_{I_c}$ for the potential~$-t \log |Df_c|$}, if the supremum above is attained at this measure.

Similarly, denote by~$\sM_c^{\C}$ the space of probability measures supported on~$J_c$ that are invariant by~$f_c$ and for a measure~$\mu$ in~$\sM_c^{\C}$ denote by~$h_\mu(f_c)$ the measure-theoretic entropy of~$f_c$ with respect to~$\mu$.
The \emph{geometric pressure function $P_c^{\C} : \R \to \R$ of~$f_c|_{J_c}$} is defined by
$$ P_c^{\C}(t)
\=
\sup \left\{ h_\mu(f_c) - t \int \log |Df_c| \ d\mu \mid \mu \in \sM_c^{\C} \right\}. $$
An invariant probability measure supported on~$J_c$ is an \emph{equilibrium state of~$f_c|_{J_c}$ for the potential~$- t \log |Df_c|$} if the supremum above is attained at this measure.

Following the usual terminology in statistical mechanics, for a given~$t_* > 0$ the map~$f_c|_{I_c}$ (resp.~$f_c|_{J_c}$) has a \emph{phase transition at~$t_*$} if~$P_c^{\R}$ (resp.~$P_c^{\C}$) is not real analytic at~$t = t_*$.
As mentioned above, if~$c \neq -2$ and if~$f_c$ is not uniformly hyperbolic and does not satisfy the Collet-Eckmann condition, then~$f_c|_{I_c}$ (resp.~$f_c|_{J_c}$) has a phase transition at the first zero of~$P_c^{\R}$ (resp.~$P_c^{\C}$) and it has no other phase transitions.
In accordance with the usual interpretation of~$t > 0$ as the inverse of the temperature in statistical mechanics, we say that such a phase transition is of \emph{high temperature}.
For a real parameter~$c$ and for~$t_* > 0$ the map~$f_c|_{I_c}$ (resp. $f_c|_{J_c}$) has a \emph{low-temperature phase transition at~$t_*$}, if it has a phase transition at~$t_*$ and~$P_c^{\R}(t_*) < 0$ (resp. $P_c^{\C}(t_*) < 0$).
Note that if~$f_c|_{I_c}$ (resp.~$f_c|_{J_c}$) has a low-temperature phase transition, then~$f_c$ satisfies the Collet-Eckmann condition.
\begin{generic}[Main Theorem]
There is a real parameter~$c$ such that the critical point of~$f_c$ is non-recurrent and such that both, $f_c|_{I_c}$ and~$f_c|_{J_c}$ have a low-temperature phase transition.
Furthermore, the parameter~$c$ can be chosen so that each of the functions~$P_c^{\R}$ and~$P_c^{\C}$ is non-differentiable at the phase transition and so that each of the maps~$f_c|_{I_c}$ and~$f_c|_{J_c}$ has a unique equilibrium state at the phase transition.
\end{generic}
For the parameter~$c$ we use to prove the Main Theorem, we show that the equilibrium state at the phase transition is ergodic and mixing and that its measure-theoretic entropy is strictly positive, see Proposition~\ref{p:first-order phase transition} in~\S\ref{s:reduced statement}.
Combined with results of Young~\cite{You99} and Gouezel~\cite[\emph{Th{\'e}or{\`e}me}~$2.3.1$]{gouezelthesis}, our estimates imply that the decay of correlations of this measure is (at most) stretch exponential.

In the companion paper~\cite{CorRivb}, we use the results of this paper to show that there is a real parameter~$c$ and~$t_* > 0$ such that both, $f_c|_{I_c}$ and~$f_c|_{J_c}$ have a high-order phase transition at~$t = t_*$ and such that the functions~$P_c^{\R}$ and~$P_c^{\C}$ are bounded from above and from below by smooth functions that coincide at~$t = t_*$.
In that case there is no equilibrium state at~$t = t_*$, see~\cite[Corollary~$1.3$]{InoRiv12}.

\begin{rema}
\label{r:non-analytic spectra}
For a parameter~$c$ in~$\C$ the dimension spectrum for Lyapunov exponents of the complex quadratic polynomial~$f_c(z) = z^2 + c$ is essentially the Legendre transform of~$P_c^{\C}(t)$, see~\cite[Theorem~$1$]{GelPrzRam10} for a precise statement and~\cite{Pes97} for the general theory.
So, for a complex quadratic polynomial as in the Main Theorem the dimension spectrum for Lyapunov exponents is not real analytic.\footnote{The following argument shows that for~$c$ as in the Main Theorem, the Legendre transform of~$P_c^{\C}$ is not real analytic.
Since~$P_c^{\C}$ is not differentiable at the phase transition, there is an interval on which the Legendre transform of~$P_c^{\C}$ is affine.
So, if the Legendre transform was real analytic, then it would be affine on all of its domain of definition.
This can only happen if~$P_c^{\C}$ is affine up to the phase transition.
But \cite[Theorem~C]{MakSmi00} or \cite[Theorem~D]{PrzRiv11} imply that this is not the case.}
A similar behavior is expected for the dimension spectrum for Lyapunov exponents of an interval map as in the Main Theorem.\footnote{More precisely, we expect the dimension spectrum of Lyapunov exponents not to be real analytic at the left end point of the interval~$A$ appearing in~\cite[Theorem~A]{IomTod11}.}
The Legendre transform of~$P_c^{\R}$ (or~$P_c^{\C}$) is also related to the dimension spectrum for pointwise dimension and the rate function in certain large deviation principles; see for example~\cite[\S$5$]{MakSmi00} for the former and~\cite[Theorem~$1.2$ or~$1.3$]{KelNow92} and~\cite[Corollary~B.$4$]{PrzRiv11} for the latter.
So for a map as in the Main Theorem we expect these functions not to be real analytic either.
Finally, note that in the complex setting the integral means spectrum associated to~$f_c$ is an affine function of~$P_c^{\C}$, see~\cite[Lemma~$2$]{BinMakSmi03}.
So for a parameter~$c$ as in the Main Theorem the integral means spectrum associated to~$f_c$ is not real analytic.
\end{rema}

\subsection{Organization}
\label{ss:organization}
After recalling some well-known facts in~\S\ref{s:preliminaries}, we define in~\S\ref{s:parameter plane} the set of parameters~$\bigcup_{n = 3}^{+ \infty} \cK_n$, from which we choose the parameter fulfilling the properties in the Main Theorem.
In~\S\S\ref{s:parameter plane}, \ref{ss:uniform distortion bound} we show various combinatorial properties of the corresponding quadratic maps, as well as some distortion bounds and other preliminary estimates.
For~$n \ge 3$ and~$c$ in~$\cK_n$, the integer~$n$ indicates the time the forward orbit of~$c$ under~$f_c$ takes for entering a certain Cantor set~$\Lambda_c$ that is invariant by~$f_c^3$, see~\S\ref{ss:expanding Cantor set} for the definition of~$\Lambda_c$ and some of its properties.
The map~$f_c^3|_{\Lambda_c}$ is uniformly expanding and conjugated to the shift map acting on~$\{0, 1\}^{\N_0}$.
The set~$\cK_n$ is such that the function that to each~$c$ in~$\cK_n$ associates the itinerary of~$f_c^n(c)$ in~$\Lambda_c$ under~$f_c^3|_{\Lambda_c}$, is a bijection (Proposition~\ref{p:ps}).
Thus, within~$\cK_n$, we can uniquely prescribe the itinerary of the postcritical orbit.

For~$n \ge 3$ and~$c$ in~$\cK_n$, the map~$f_c^3$ has precisely~$2$ fixed points in~$\Lambda_c$, denoted by~$p(c)$ and~$\wtp(c)$.
They correspond to the symbols~$0$ and~$1$, respectively.
For large~$n$ and every~$c$ in~$\cK_n$, the derivative\ of~$f_c^3$ at~$p(c)$ is strictly larger than that at~$\wtp(c)$, see Appendix~\ref{s:appendix} and Proposition~\ref{p:ps}.
Similarly as in the example of Makarov and Smirnov, we consider a parameter~$c$ such that for every large integer~$k \ge 1$, the forward orbit of~$c$ under~$f_c$ up to a time~$k$ spends roughly~$\sqrt{k}$ of the time in the branch of~$f_c^3|_{\Lambda_c}$ corresponding to~$p(c)$ (of symbol~$0$), and the rest of the time in the other branch (of symbol~$1$).
An additional difficulty in our situation is that the map~$f_c^3|_{\Lambda_c}$ is non-linear, and thus in our estimates we have to deal with additional distortion terms.
We overcome this difficulty, in part, by choosing an itinerary having only large blocks of~$0$'s and~$1$'s, see Lemma~\ref{l:phased itinerary} for the precise definition of the itinerary.
Choosing~$n$ large also help us to overcome this difficulty.
Roughly speaking, in the example of Makarov and Smirnov this last choice corresponds to taking a small critical branch.\footnote{This is not entirely accurate, but it is a good first approximation.
By choosing~$n$ large we are essentially forced to consider the first return map to a smaller neighborhood of the critical point, and thus we have to deal with a larger set of orbits that never enter this set.
These extra orbits are not present in the example of Makarov and Smirnov.}

A step in proving that for a parameter~$c$ as above the geometric pressure function is not real analytic on all of~$(0, + \infty)$, is to show that this function is larger than or equal to~$t \mapsto - t \chicrit/2$.
We do this by exhibiting a sequence of periodic orbits whose Lyapunov exponents converge to~$\chicrit /2$, see~\S\ref{ss:critical line}.
The bulk of the proof of the Main Theorem, in \S\S\ref{ss:proof of Main Theorem}--\ref{s:estimating pressure}, is devoted to show that for a large value of~$t > 0$ the geometric pressure is less than or equal to~$- t \chicrit/2$.
This implies that the geometric pressure is in fact equal to~$- t \chicrit/2$, and therefore that the geometric pressure function coincides with the function~$t \mapsto - t \chicrit/2$ on some (right) half line.
Since at~$t = 0$ the geometric pressure is equal to the topological entropy and it is therefore strictly positive, it follows that the geometric pressure function cannot be real analytic on all of~$(0, + \infty)$.

To prove that for a large value of~$t > 0$ the geometric pressure is less than or equal to~$- t \chicrit/2$, we show, as in the example of Makarov and Smirnov, that the pressure function can be estimated using a certain ``postcritical series'', defined solely in terms of the derivatives of~$f_c$ along the forward orbit of~$c$ (Proposition~\ref{p:improved MS criterion} in~\S\ref{s:estimating pressure}).
To make this estimate we proceed in a different way than the example of Makarov and Smirnov.
We consider the pressure function as defined through the tree of preimages of the critical point.
An important step of the proof is to show that the dynamics is sufficiently expanding far away from the critical point (Proposition~\ref{p:landing derivatives} in~\S\ref{s:landing derivatives}), and thus that the geometric pressure is governed by those backward orbits that visit a given neighborhood of the critical point.
For a conveniently chosen neighborhood~$V_c$ of the critical point, we estimate the pressure of the backward orbits of the critical point that visit~$V_c$ using the first return map~$F_c$ of~$f_c$ to~$V_c$, and a certain~$2$ variables pressure function of~$F_c$.
This last pressure function depends on the geometric potential of~$F_c$ and the first return time function.
The neighborhood~$V_c$ and the first return map~$F_c$ are defined in~\S\ref{ss:induced map}, and the~$2$ variables pressure function of~$F_c$ is defined in~\S\ref{ss:Bowen type formula}.
The connection between the geometric pressure of~$f_c$ and the~$2$ variables pressure function of~$F_c$ is through a Bowen type formula that we state as Proposition~\ref{p:Bowen type formula} in~\ref{ss:Bowen type formula}.

The~$2$ variables pressure function of~$F_c$ is defined through a subadditive sequence in a standard way, see~\S\ref{ss:Bowen type formula}.
Thanks to the fact that our distortion bounds are independent of~$n$ and of~$c$ in~$\cK_n$, the first term of the subadditive sequence provides an estimate of the~$2$ variable pressure that is good enough for our purposes.
To estimate the first term of the subadditive sequence, in \S\ref{s:estimating pressure} we partition the components of the domain of~$F_c$ into ``levels'', according to the first return time to a certain neighborhood of~$\Lambda_c$.
The proof of Proposition~\ref{p:improved MS criterion} consists of showing that for each integer~$k \ge 0$, the contribution of the components of the domain of~$F_c$ of level~$k$ is equal to the~$k$-th term of the postcritical series, up to a multiplicative constant, see Lemma~\ref{l:estimating Z_1 by the postcritical series}.

We state a consequence of Proposition~\ref{p:improved MS criterion} as Proposition~\ref{p:first-order phase transition} in~\S\ref{s:reduced statement}, from which we deduce the Main Theorem in~\S\ref{ss:proof of Main Theorem}.
The proof of Proposition~\ref{p:first-order phase transition} is given in~\S\ref{s:estimating pressure}, after the proof of Proposition~\ref{p:improved MS criterion}.

\subsection{Acknowledgments}
We thank Weixiao Shen for a useful remark.
The first named author acknowledges partial support from FONDECYT grant 11121453, Anillo DYSYRF grant ACT 1103, MATH-AmSud grant DYSTIL, and Basal-Grant CMM PFB-03.
This article was completed while the second named author was visiting Brown University and the Institute for Computational and Experimental Research in Mathematics (ICERM).
He thanks both of these institutions for the optimal working conditions provided, and acknowledges partial support from FONDECYT grant 1100922.
The figures in this paper were made with Wolf Jung's program ``Mandel''. 
\section{Preliminaries}
\label{s:preliminaries}
We use~$\N$ to denote the set of integers that are greater than or equal to~$1$ and put~$\N_0 \= \N \cup \{ 0 \}$.

For an annulus~$A$ contained in~$\C$ we use~$\modulus(A)$ to denote the conformal modulus of~$A$.
\subsection{Koebe principle}
\label{ss:Koebe}
We use the following version of Koebe distortion theorem that can be found, for example, in~\cite{McM94}.
Given an open subset~$G$ of~$\C$ and a biholomorphic map~$f : G \to \C$, the \emph{distortion of~$f$} on a subset~$C$ of~$G$ is
$$ \sup_{x, y \in C} |Df(x)|/|Df(y)|. $$
\begin{generic}[Koebe Distortion Theorem]
For each~$A > 0$ there is a constant~$\Delta > 1$ such that for each topological disk~$\hW$ contained in~$\C$ and each compact set~$K$ contained in~$\hW$ and such that~$\hW \setminus K$ is an annulus of modulus at least~$A$, the following property holds: For each open topological disk~$U$ contained in~$\C$ and every biholomorphic map~$f : U \to \hW$, the distortion of~$f$ on~$f^{-1}(K)$ is bounded by~$\Delta$.
\end{generic}

\subsection{Quadratic polynomials, Green functions and B{\"o}ttcher coordinates}
\label{ss:quadratic polynomials}
In this subsection and the next we recall some basic facts about the dynamics of complex quadratic polynomials, see for instance~\cite{CarGam93} or~\cite{Mil06} for references.

For~$c$ in~$\C$ denote by~$f_c$ the complex quadratic polynomial
$$ f_c(z) = z^2 + c, $$
and by~$K_c$ the \emph{filled Julia set} of $f_c$; that is, the set of all points~$z$ in~$\C$ whose forward orbit under~$f_c$ is bounded in~$\C$.
The set~$K_c$ is compact and its complement is the connected set consisting of all points whose orbit converges to infinity in the Riemann sphere.
Furthermore, we have $f_c^{-1}(K_c) = K_c$ and~$f_c(K_c) = K_c$. 
The boundary~$J_c$ of~$K_c$ is the \emph{Julia set of~$f_c$}.

For a parameter~$c$ in~$\C$, the \emph{Green function of~$K_c$} is the function $G_c:\C \to [0,+\infty)$ that is identically~$0$ on~$K_c$ and that for~$z$ outside~$K_c$ is given by the limit
\begin{equation}
\label{def:Green function}
  G_c(z) = \lim_{n\rightarrow +\infty} \frac{1}{2^n} \log |f_c^n(z)| > 0.
\end{equation}
The function~$G_c$ is continuous, subharmonic, satisfies~$G_c \circ f_c = 2 G_c$ on~$\C$, and it is harmonic and strictly positive outside~$K_c$.
On the other hand, the critical values are bounded from above by~$G_c(0)$ and the open set
$$ U_c \= \{z\in \C \mid G_c(z) > G_c(0)\} $$
is homeomorphic to a punctured disk.
Notice that $G_c(c)=2G_c(0)$, thus~$U_c$ contains~$c$. 
Moreover, the Green's functions of~$f_c$ and~$f_{\overline{c}}$ are related by $G_{\overline{c}}(z) = G_c(\overline{z})$.

By B{\"o}ttcher's Theorem there is a unique conformal representation
\[
 \varphi_c: U_c
\rightarrow
\{z\in \C \mid |z| > \exp (G_c(0)) \},
\]
that conjugates~$f_c$ to $z \mapsto z^2$.
It is called \emph{the B{\"o}ttcher coordinate of~$f_c$} and satisfies $G_c = \log |\varphi_c|$.
Note that~$U_{\overline{c}} = \overline{U_c}$ and~$\varphi_{\overline{c}} = \overline{\varphi_c}$.

\subsection{External rays and equipotentials}
\label{ss:rays and equipotentials}
Let~$c$ be in~$\C$.
For~$v > 0$ the \emph{equipotential~$v$ of~$f_c$} is by definition~$G_c^{-1}(v)$.
A \emph{Green's line of~$G_c$} is a smooth curve on the complement of~$K_c$ in~$\C$ that is orthogonal to the equipotentials of~$G_c$ and that is maximal with this property. 
Given~$t$ in~$\R / \Z$, the \emph{external ray of angle~$t$ of~$f_c$}, denoted by~$R_c(t)$, is the Green's line of~$G_c$ containing
$$ \{ \varphi_c^{-1}(r \exp(2 \pi i t)) \mid \exp(G_c(0))< r < +\infty \}. $$
By the identity~$G_c \circ f_c= 2G_c$, for each~$v > 0$ and each~$t$ in~$\R / \Z$ the map~$f_c$ maps the equipotential~$v$ to the equipotential~$2v$ and maps~$R_c(t)$ to~$R_c(2t)$.
For~$t$ in~$\R / \Z$ the external ray~$R_c(t)$ \emph{lands at a point~$z$}, if~$G_c : R_c(t) \to (0, + \infty)$ is a bijection and if~$G_c|_{R_c(t)}^{-1}(v)$ converges to~$z$ as~$v$ converges to~$0$ in~$(0, + \infty)$.
By the continuity of~$G_c$, every landing point is in $J_c = \partial K_c$.

We use the following general lemma several times.
\begin{lemm}\label{lem:0}
Let~$c$ be a parameter in~$\C$, let~$t$ be in~$\R / \Z$ and suppose that the external ray~$R_c(t)$ lands at a point~$z_0$ of~$K_c$ different from~$c$; so~$f_c^{-1}(z_0)$ consists of~$2$ distinct points.
Then each point of~$f_c^{-1}(z_0)$ is the landing point of precisely~$1$ of the external rays~$R_c(t/2)$ or~$R_c((t + 1)/2)$.
\end{lemm}
\proof
Since~$f_c^{-1}(z_0)$ consists of~$2$ distinct points, it is enough to show that each point~$z$ of~$f_c^{-1}(z_0)$ is the landing point of either~$R_c(t/2)$ or~$R_c((t + 1)/2)$.
Since~$z_0$ is different from~$c$, there is an open neighborhood~$U$ of~$z$ and an open neighborhood~$U_0$ of~$z_0$ such that~$f_c$ maps~$U$ biholomorphically to~$U_0$.
Reducing~$U$ and~$U_0$ if necessary, it follows that $f_c^{-1}(R_c(t))$ is contained in an external ray landing at~$z$.
It must be either~$R_c(t/2)$ or~$R_c((t + 1)/2)$.
\endproof

The \emph{Mandelbrot set~$\cM$} is the subset of~$\C$ of those parameters~$c$ for which~$K_c$ is connected.
The function
\[
 \begin{array}{cccl}
  \Phi : &\C \setminus \cM & \to & \C \setminus \cl{\D}\\
          &     c         &  \mapsto           & \Phi(c) \= \varphi_c(c)
 \end{array}
\]
is a conformal representation, see~\cite[VIII, \emph{Th{\'e}or{\`e}me}~$1$]{DouHub84}.
Since for each parameter~$c$ in~$\C$ we have~$\varphi_{\overline{c}} = \overline{\varphi_c}$, it follows that~$\Phi(\overline{c}) = \overline{\Phi(c)}$; that is, $\Phi$ is real.
For~$v > 0$ the \emph{equipotential~$v$ of~$\cM$} is by definition
$$ \cE(v) \= \Phi^{-1}(\{z\in \C \mid |z| = v \}). $$ 
On the other hand, for~$t$ in~$\R / \Z$ the set
$$ \cR(t) \= \Phi^{-1}(\{r \exp(2 \pi i t) \mid r > 1 \}). $$
is called the \emph{external ray of angle~$t$ of~$\cM$}.
We say that $\cR(t)$ \emph{lands at a point~$z$} in~$\C$ if~$\Phi^{-1} (r \exp(2\pi i t))$ converges to~$z$ as $r \searrow 1$.
When this happens~$z$ belongs to~$\partial \cM$.

\subsection{The wake~$1/2$}
\label{ss:wake 1/2}
In this subsection we recall a few facts that can be found for example in~\cite{DouHub84} or~\cite{Mil00c}.

Both external rays~$\cR(1/3)$ and~$\cR(2/3)$ of~$\cM$ land at the parameter~$c = -3/4$ and these are the only external rays of~$\cM$ that land at this point, see for example~\cite[Theorem~$1.2$]{Mil00c}.
In particular, the complement in~$\C$ of the set
$$ \cR(1/3) \cup \cR(2/3) \cup \{ - 3/4 \} $$
has~$2$ connected components; denote by~$\cW$ the connected component containing the point~$c = -2$ of~$\cM$.

For each parameter~$c$ in~$\cW$ the map~$f_c$ has~$2$ distinct fixed points; one of the them is the landing point of the external ray~$R_c(0)$ and it is denoted by~$\beta(c)$; the other one is denoted by~$\alpha(c)$.
The only external ray landing at~$\beta(c)$ is~$R_c(0)$.
Lemma~\ref{lem:0} implies that the only external ray landing at~$-\beta(c)$ is~$R_c(1/2)$.

For the following fact, see for example~\cite[Theorem~$1.2$]{Mil00c}.
\begin{theo}
\label{theo:wake 1/2}
Let~$c$ be a parameter in~$\cW$.
Then the only external rays of~$f_c$ landing at~$\alpha(c)$ are~$R_c(1/3)$ and~$R_c(2/3)$.
\end{theo}
For~$c$ in~$\cW$, the complement of~$R_c(1/3) \cup R_c(2/3) \cup \{ \alpha(c) \}$ in~$\C$ has~$2$ connected components; on containing~$- \beta(c)$ and~$z = c$, and the other one containing~$\beta(c)$ and~$z = 0$.
On the other hand, the point~$\alpha(c)$ has~$2$ preimages by~$f_c$: Itself and~$\talpha(c)
\= - \alpha(c)$.
Together with Lemma~\ref{lem:0}, the theorem above implies that~$R_c(1/6)$ and~$R_c(5/6)$ are the only external rays landing at~$\talpha(c)$.

\begin{theo}[\cite{DouHub84}, VIII, \emph{Th{\'e}or{\`e}me}~$2$ and XIII, \S$1$]
\label{DHrays}
Let~$p$ and~$q$ be integers without common factors, with~$q$ even.
Then the external ray~$\cR(p/q)$ of~$\cM$ lands and the landing point~$c$ is such that the critical point of~$f_c$ is eventually periodic but not periodic and such that the critical value~$c$ of~$f_c$ is the landing point of the external ray~$R_c(p/q)$ of~$f_c$.
Conversely, if~$c$ is a parameter in~$\C$ such that the critical point of~$f_c$ is eventually periodic but not periodic, then there are integers~$p$ and~$q$ without common factors and with~$q$ even, such that the critical value~$c$ of~$f_c$ is the landing point of~$R_c(p/q)$; moreover, every external ray of~$f_c$ landing at~$c$ is of this form.
In this case the parameter~$c$ is the landing point of the external ray~$\cR_c(p/q)$ of~$\cM$.
\end{theo}
Note that for the parameter~$c = -2$ we have~$c = - \beta(c)$, so the theorem above implies that~$\cR(1/2)$ is the only external ray of~$\cM$ that lands at~$- 2$.

\subsection{Yoccoz puzzles and para-puzzle}
\label{ss:puzzles}
In this subsection we recall the definitions of Yoccoz puzzle and para-puzzle.
We follow \cite{Roe00}.

\begin{defi}[Yoccoz puzzles]
Fix~$c$ in~$\cW$ and consider the open region $X_c \= \{z\in \C \mid G_c(z) <
1\}$. 
The \emph{Yoccoz puzzle of~$f_c$} is given by the following sequence of
graphs~$(I_{c, n})_{n = 0}^{+ \infty}$ defined for~$n = 0$ by:
\[
 I_{c,0} \= \partial X_c \cup (X_c \cap \cl{R_c(1/3)} \cap \cl{R_c(2/3)})
\]
and for~$n \ge 1$ by~$I_{c,n} \= f_c^{-n}(I_{c,0})$.
The \emph{puzzle pieces of depth~$n$} are the connected components of $f_c^{-n}(X_c) \setminus I_{c,n}$.
The puzzle piece of depth~$n$ containing a point~$z$ is denoted by~$P_{c,n}(z)$.
\end{defi}

Note that for a real parameter~$c$, every puzzle piece intersecting the real line is invariant under complex conjugation.
Since puzzle pieces are simply-connected, it follows that the intersection of such a puzzle piece with~$\R$ is an interval.

\begin{defi}[Yoccoz para-puzzles\footnote{In contrast with~\cite{Roe00}, we only consider para-puzzles contained in~$\cW$.}]
Given an integer~$n \ge 0$ put
$$ J_n
\=
\{t\in [1/3,2/3] \mid 2^n t ~ (\mathrm{mod}\, 1) \in \{1/3,2/3\} \}, $$
let~$\cX_n$ be the intersection of~$\cW$ with the open region in the parameter plane bounded by the equipotential~$\cE(2^{-n})$ of~$\cM$, and put
\[
 \cI_{n}
\=
\partial \cX_n \cup \left( \cX_n \cap \bigcup_{t\in J_n} \cl{\cR(t)} \right).
\]
Then the \emph{Yoccoz para-puzzle of~$\cW$} is the sequence of graphs~$(\cI_n)_{n = 0}^{+ \infty}$.
The \emph{para-puzzle pieces of depth~$n$} are the connected components of $\cX_n \setminus \cI_n$.
The para-puzzle piece of depth~$n$ containing a parameter~$c$ is denoted by~$\cP_n(c)$.  
\end{defi}
Observe that there is only~$1$ para-puzzle piece of depth~$0$ and only~$1$ para-puzzle piece of depth~$1$; they are bounded by the same external rays but different equipotentials.
Both of them contain~$c = - 2$.

\begin{defi}[Holomorphic motion] 
Let~$\cC$ be a complex manifold and fix~$c_0$ in~$\cC$.
Given a subset~$Z$ of~$\C$, a map
$$ h : \cC \times Z \to \cC \times \C $$
of the form $(c, z) \mapsto (c, h^c(z))$ is a \emph{holomorphic motion based at~$c_0$} if
$h^{c_0}$ is the identity on~$Z$, if for each~$z$ in~$Z$ its restriction to
$\cC \times \{ z \}$ is holomorphic and if for each $c \in \cC$
its restriction to $\{ c \} \times Z$ is injective.
\end{defi}

See~\cite{Roe00} for a reference to the following lemma; the statement here is slightly different from the statement in~\cite{Roe00} since we extend the domain of definition of the holomorphic motions, but the proof is the same. 
For each integer~$n \ge 1$, put
$$ V_n \= \{w \in \C \mid \log^+|w|\ge 2^{-n}\}. $$
\begin{lemm}
\label{lem:hm}
Let~$n \ge 0$ be an integer and~$c_0$ a parameter contained in a para-puzzle of
depth~$n$.
Then there exists a holomorphic motion
\[
\begin{array}{cccl}
 h_n:  & \cP_n(c_0) \times (I_{c_0,n+1}\cup \varphi_{c_0}^{-1}(V_{n + 1})) &
\to & \cP_n(c_0) \times \C \\
       &            (c,z)               & \mapsto & h_n(c,z) \= (c, h_n^c(z))
\end{array} 
\]
such that for every~$c$ in~$\cP_n(c_0)$ the function $h_n^c$ is an extension of the
restriction of $\varphi_c^{-1}\circ \varphi_{c_0}$ to $I_{c_0,n+1} \cup \varphi_{c_0}^{-1}(V_{n + 1})$ that satisfies~$I_{c,n+1} = h_n^c(I_{c_0,n+1})$. 
Moreover, when $n\ge 1$ the map~$h_n$ coincides with~$h_{n-1}$ on $\cP_n(c_0) \times (I_{c_0,n} \cup \varphi_{c_0}^{-1}(V_n))$ and for each~$c$ in~$\cP_n(c_0)$
we have $f_c \circ h_n^c = h_{n - 1}^c \circ f_{c_0}$ on~$I_{c_0,n+1} \cup \varphi_{c_0}^{-1}(V_{n + 1})$.
\end{lemm}

\section{Parameters}
\label{s:parameter plane}
In this section we study the set of parameters from which we choose the parameter in the Main Theorem and at the same time introduce some notation used in the rest of the paper.

Given an integer~$n \ge 3$, let~$\cK_n$ be the set of all those real parameters~$c$ such that the following properties hold.
\begin{enumerate}
\item[1.]
We have $c < 0$ and for each~$j$ in~$\{1, \ldots, n - 1 \}$ we have~$f_c^j(c) > 0$.
\item[2.]
For every integer~$k \ge 0$ we have
$$ f_c^{n + 3k + 1}(c) < 0
\text{ and }
f_c^{n + 3k + 2}(c) > 0. $$
\end{enumerate}

Note that for a parameter~$c$ in~$\cK_n$ the critical point of~$f_c$ cannot be asymptotic to a non-repelling periodic point,  see~\cite[\S$8$]{MilThu88}.
This implies that all the periodic points of~$f_c$ in~$\C$ are hyperbolic repelling and therefore that~$K_c = J_c$, see~\cite{Mil06}.
On the other hand, we have~$f_c(c) > c$ and the interval~$I_c = [c, f_c(c)]$ is invariant by~$f_c$.
This implies that~$I_c$ is contained in~$J_c$ and hence that for every real number~$t$ we have~$P_c^{\R}(t) \le P_c^{\C}(t)$.
Note also that~$f_c|_{I_c}$ is not renormalizable, so~$f_c$ is topologically exact on~$I_c$, see for example~\cite[Thoerem~III.$4$.$1$]{dMevSt93}.

Since for~$c$ in~$\cK_n$ the critical point of~$f_c$ is not periodic, we can define the sequence~$\iota(c)$ in~$\{0, 1 \}^{\N_0}$ for each~$k \ge 0$ by
$$ \iota(c)_k
\=
\begin{cases}
0 & \text{ if }  f_c^{n + 3k}(c) < 0; \\
1  & \text{ if } f_c^{n + 3k}(c) > 0.
\end{cases} $$

The remainder of this section is devoted to prove the following proposition.
\begin{prop}
\label{p:ps}
For each integer~$n \ge 3$ the set~$\cK_n$ is a compact subset of
$$ \cP_n(-2) \cap (-2, -3/4) $$
and for every sequence~$\underline{x}$ in~$\{0,1\}^{\N_0}$ there is a unique parameter~$c$ in~$\cK_n$ such that~$\iota(c) = \underline{x}$.
Finally, for each~$\delta > 0$ there is~$n_1 \ge 3$ such that for each integer~$n \ge n_1$ the set~$\cK_n$ is contained in the interval~$(-2, -2 + \delta)$.
\end{prop}

After defining some sequences of puzzle pieces that are important for the rest of this paper in~\S\ref{ss:landing to central}, we study the para-puzzle pieces containing~$c = -2$ in~\S\ref{ss:para-puzzle nest} and the maximal invariant set of~$f_c^3$ in~$P_{c, 1}(0)$ in~\S\ref{ss:expanding Cantor set}.
The proof of Proposition~\ref{p:ps} is in~\S\ref{ss:para-puzzle nest shrinks}.

\subsection{First landing domains to $P_{c, 1}(0)$}
\label{ss:landing to central}
Fix a parameter~$c$ in~$\cP_0(-2)$.

The following are consequences of the facts recalled in~\S\ref{ss:wake 1/2}.
There are precisely~$2$ puzzle pieces of depth~$0$: $P_{c, 0}(\beta(c))$ and~$P_{c, 0}(-\beta(c))$.
Each of them is bounded by the equipotential~$1$ and by the closures of the external rays landing at~$\alpha(c)$.
Furthermore, the critical value~$c$ of~$f_c$ is contained in~$P_{c, 0}(- \beta(c))$ and the critical point in~$P_{c, 0}(\beta(c))$.
It follows that the set~$f_c^{-1}(P_{c, 0}(\beta(c)))$ is the disjoint union of~$P_{c, 1}(- \beta(c))$ and~$P_{c, 1}(\beta(c))$, so~$f_c$ maps each of the sets~$P_{c, 1}(-\beta(c))$ and~$P_{c, 1}(\beta(c))$ biholomorphically to~$P_{c, 0}(\beta(c))$.
Moreover, there are precisely~$3$ puzzle pieces of depth~$1$: 
$$ P_{c, 1}(-\beta(c)),
P_{c, 1}(0)
\text{ and }
P_{c, 1}(\beta(c)); $$
$P_{c, 1}(- \beta(c))$ is bounded by the equipotential~$1/2$ and by the closures of the external rays that land at~$\alpha(c)$; $P_{c, 1}(\beta(c))$ is bounded by the equipotential~$1/2$ and by the closures of the external rays that land at~$\talpha(c)$; and~$P_{c, 1}(0)$ is bounded by the equipotential~$1/2$ and by the closures of the external rays that land at~$\alpha(c)$ and at~$\talpha(c)$.
In particular, the closure of~$P_{c, 1}(\beta(c))$ is contained in~$P_{c, 0}(\beta(c))$.

Put
$$ \phi_c \= f_c|_{P_{c, 1}(- \beta(c))}^{-1}
\text{ and }
\tphi_c \= f_c|_{P_{c, 1}(\beta(c))}^{-1}. $$
Since the closure of~$\tphi_c(P_{c, 0}(\beta(c))) = P_{c, 1}(\beta(c))$ is contained in~$P_{c, 0}(\beta(c))$, all the iterates of~$\tphi_c$ are defined on~$P_{c, 0}(\beta(c))$ and take images in~$P_{c, 1}(\beta(c))$.
Put $\alpha_0(c) \= \alpha(c)$, $\talpha_0(c) \= \talpha(c)$ and for each integer~$n \ge 1$ put
\[
\talpha_n(c) \= \tphi_c^n(\talpha_0(c)),
\alpha_n(c) \= \phi_c(\talpha_{n - 1}(c)),
\]
$$\tV_{c, n} \= \tphi_c^n(P_{c, 1}(0)),
\text{ and }
V_{c, n} \= \phi_c \circ \tphi_c^{n - 1} (P_{c, 1}(0)). $$
Note that~$f_c^n$ maps each of the sets~$V_{c, n}$ and~$\tV_{c, n}$ biholomorphically to~$P_{c, 1}(0)$; 
so both, $V_{c, n}$ and~$\tV_{c, n}$ are puzzle pieces of depth~$n + 1$.
On the other hand, $f_c^n$ maps each of the sets~$P_{c, n}(- \beta(c))$ and~$P_{c, n}(\beta(c))$ biholomorphically to~$P_{c, 0}(\beta(c))$.
Note moreover that if~$c$ is real, then~$f_c$, $\phi_c$ and~$\tphi_c$ are all real, so~$\alpha(c)$, $\talpha(c)$, $\alpha_n(c)$ and $\talpha_n(c)$ are also real and each of the sets~$V_{c, n}$, $\tV_{c, n}$, $P_{c, n}(- \beta(c))$ and~$P_{c, n}(\beta(c))$ is invariant by complex conjugation and intersects~$\R$.

\begin{lemm}
\label{lem:puzzle pieces}
Let~$c$ be a parameter in~$\cP_0(-2)$.
Then for every integer $n \ge 0$ the only external rays that land at~$\alpha_n(c)$ are~$R_c\left(\frac{3\cdot 2^{n}-1}{3\cdot 2^{n+1}}\right)$ and~$R_c\left(\frac{3\cdot 2^{n}+1}{3\cdot 2^{n+1}}\right)$ and the only external rays that land at~$\talpha_n(c)$ are $R_c\left(\frac{1}{3\cdot 2^{n+1}}\right)$ and $R_c\left(\frac{3\cdot 2^{n+1}-1}{3\cdot 2^{n+1}}\right)$.
Furthermore, for each integer~$n \ge 1$ the following properties hold.
\begin{enumerate}
\item[1.]
The only puzzle pieces of depth~$n + 1$ contained in~$P_{c, n}(- \beta(c))$ are~$P_{c, n + 1}(-\beta(c))$ and~$V_{c, n}$.
Moreover, the closure of~$P_{c, n + 1}(- \beta(c))$ is contained in the open set~$P_{c, n}(- \beta(c))$.
\item[2.]
The puzzle piece~$P_{c,n}(-\beta(c))$ is bounded by the external rays landing at~$\alpha_{n-1}(c)$ and the equipotential~$1/2^{n}$; the puzzle piece~$P_{c,n}(\beta(c))$ is bounded by the closure of the external rays landing at~$\talpha_{n-1}(c)$ and the equipotential~$1/2^{n}$.
\end{enumerate}
\end{lemm}
\proof
For an integer~$n \ge 0$ put $\theta_n \= \frac{1}{3\cdot 2^{n+1}}$ and~$\theta_n' \= 1 - \theta_n$.

The proof of the first assertion is by induction.
When~$n = 0$ the assertion is shown in~\S\ref{ss:wake 1/2}.
Given an integer~$n \ge 0$ assume that the only external rays that land at~$\talpha_n(c)$ are those of angles~$\theta_n$ and~$\theta_n'$.
Since~$f_c^{-1}(\talpha_n(c)) = \{ \alpha_{n+1}(c), \talpha_{n+1}(c) \}$, by Lemma~\ref{lem:0} the only external rays landing at~$\alpha_{n + 1}(c)$ or~$\talpha_{n + 1}(c)$ are those of angles $\frac{\theta_n}{2}$, $\frac{\theta_n + 1}{2}$, $\frac{\theta_n'}{2}$ and~$\frac{\theta_n' + 1}{2}$.
Since
$$ \frac{\theta_n}{2}
<
\frac{1}{6}
<
\frac{\theta_n'}{2}
<
\frac{\theta_n + 1}{2}
<
\frac{5}{6}
<
\frac{\theta_n' + 1}{2} $$
and since~$\talpha_{n+1}(c)$ is in~$P_{c,1}(\beta(c))$, the external rays of angles~$\frac{\theta_n}{2}$ and~$\frac{\theta_n' + 1}{2}$
land at $\talpha_{n+1}(c)$.
By Lemma~\ref{lem:0} it follows that the external rays of angles~$\frac{\theta_n'}{2}$ and~$\frac{\theta_n + 1}{2}$ land at~$\alpha_{n+1}(c)$.
This completes the proof of the induction step and of the first assertion of the lemma.

To prove the rest of the assertions, assume~$n \ge 1$.
Since~$f_c^n$ maps~$P_{c, n}(-\beta(c))$ biholomorphically to~$P_{c, 0}(\beta(c))$, $V_{c, n}$ biholomorphically to~$P_{c, 1}(0)$, and~$P_{c, n + 1}(\beta(c))$ biholomorphically to~$P_{c, 1}(\beta(c))$, it follows that the only puzzle pieces of depth~$n + 1$ contained in~$P_{c, n}(-\beta(c))$ are~$V_{c, n}$ and~$P_{c, n + 1}(- \beta(c))$.
On the other hand, since the closure of~$P_{c, 1}(\beta(c))$ is contained in~$P_{c, 0}(\beta(c))$, it follows that the closure of~$P_{c, n + 1}(-\beta(c))$ is contained in~$P_{c, n}(-\beta(c))$.
We have thus proved part~$1$.
When~$n = 1$ part~$2$ follows from the considerations above.
To prove part~$2$ when~$n \ge 2$, recall that~$f_c^{n - 1}$ maps each of the sets~$P_{c, n - 1}(- \beta(c))$ and~$P_{c, n - 1}(\beta(c))$ biholomorphically to~$P_{c, 0}(\beta(c))$.
Since the closure of~$P_{c, 1}(\beta(c))$ is contained in~$P_{c, 0}(\beta(c))$ and since~$\talpha(c)$ is the only point in the boundary of~$P_{c, 1}(\beta(c))$ that is in the Julia set of~$f_c$, it follows that~$\alpha_{n - 1}(c)$ is the only point in the boundary of~$P_{c, n}(-\beta(c))$ that is in the Julia set of~$f_c$ and that~$\talpha_{n - 1}(c)$ is the only point in the boundary of~$P_{c, n}(\beta(c))$ that is in the Julia set of~$f_c$.
This implies part~$2$ and completes the proof of the lemma.
\endproof

\subsection{Para-puzzle pieces containing~$c = - 2$}
\label{ss:para-puzzle nest}
The purpose of this subsection is to prove the following lemma.
\begin{lemm}
\label{lem:auxiliary para-puzzle pieces}
The following properties hold.
\begin{enumerate}
\item[1.]
For every integer~$n \ge 1$, the para-puzzle piece~$\cP_n(-2)$ contains the closure of~$\cP_{n + 1}(-2)$.
\item[2.]
For every integer~$n \ge 0$ and every parameter~$c$ in~$\cP_n(-2)$, the critical value~$c$ of~$f_c$ is in~$P_{c, n}(-\beta(c))$.
\end{enumerate}
\end{lemm}
The proof of this lemma is given after the following one.

For each integer~$n \ge 0$ put
$$ t_n \= \frac{3 \cdot 2^{n}-1}{3\cdot 2^{n+1}}
\text{ and }
t'_n \=  \frac{3\cdot 2^{n} + 1}{3\cdot 2^{n+1}}. $$
\begin{lemm}
\label{lem:para-puzzle pieces}
Fix an integer~$n \ge 1$.
Then the parameter~$c = -2$ is contained in a para-puzzle piece of depth~$n$ and there is a unique point~$\halpha_{n - 1}$ in the boundary of~$\cP_n(-2)$ that is contained in~$\cM$.
Furthermore, $\halpha_{n - 1}$ is in~$\R$, the only the external rays of~$\cM$ that land at~$\halpha_{n - 1}$ are~$\cR(t_{n - 1})$ and~$\cR(t_{n - 1}')$, and~$\cP_n(-2)$ is invariant under complex conjugation.
In particular, $\cP_n(-2)$ is bounded by the equipotential~$1/2^n$ and the closures of the external rays~$\cR(t_{n - 1})$ and~$\cR(t_{n - 1}')$ of~$\cM$.
\end{lemm}
\proof
Since~$\cR(1/2)$ is the only external ray of~$\cM$ that lands at~$c = - 2$ and since~$t = 1/2$ is not in~$J_n$, it follows that~$c = -2$ is contained in a para-puzzle of level~$n$.
On the other hand, by Theorem~\ref{DHrays} the external ray~$\cR(t_{n - 1})$ of~$\cM$ lands at a parameter in~$\cM$, denoted by~$\halpha_{n - 1}$, and the external ray~$R_{\halpha_{n - 1}}(t_{n - 1})$ of~$f_{\halpha_{n - 1}}$ lands at the critical value~$\halpha_{n - 1}$ of~$f_{\halpha_{n - 1}}$.
By Lemma~\ref{lem:puzzle pieces} we have that~$\alpha_{n - 1}(\halpha_{n - 1}) = \halpha_{n - 1}$ and that~$R_{\halpha_{n - 1}}(t_{n - 1})$ and~$R_{\halpha_{n - 1}}(t_{n - 1}')$ are the only external rays of~$f_{\halpha_{n - 1}}$ landing at~$\halpha_{n - 1}$.
Using Theorem~\ref{DHrays} again, we conclude that~$\cR(t_{n - 1})$ and~$\cR(t_{n - 1}')$ are the only external rays of~$\cM$ landing at~$\halpha_{n - 1}$.
Since~$\Phi$ is real, we have~$\cR(t_n') = \overline{\cR(t_n)}$ and therefore~$\halpha_{n - 1}$ is in~$\R$.
On the other hand, since the interval~$(t_{n - 1}, t_{n - 1}')$ is disjoint from~$J_n$ and contains~$1/2$, the closures of the external rays~$\cR(t_{n - 1})$ and~$\cR(t_{n - 1}')$ of~$\cM$ and the equipotential~$1/2^n$ of~$\cM$ bound a para-puzzle piece of depth~$n$ that contains~$c = -2$; that is, they bound~$\cP_n(-2)$.
It follows that~$\halpha_{n - 1}$ is the only point in the boundary of~$\cP_n(-2)$ in~$\cM$.
That~$\cP_n(-2)$ is invariant under complex conjugation follows from the fact that~$\halpha_{n - 1}$ and~$\Phi$ are real.
\endproof
\begin{proof}[Proof of Lemma~\ref{lem:auxiliary para-puzzle pieces}]
Part~$1$ follows from the descriptions of~$\cP_n(-2)$ and~$\cP_{n + 1}(-2)$ in terms of external rays and equipotentials given by Lemma~\ref{lem:para-puzzle pieces}.

To prove part~$2$, let~$n \ge 0$ be an integer.
We use the following direct consequence of the definitions of the puzzle and the para-puzzle and of Theorem~\ref{DHrays}: A parameter~$c$ in~$\cW$ is in a para-puzzle piece of depth~$n$ if and only if the critical value~$c$ of~$f_c$ is in a puzzle piece of depth~$n$ of~$f_c$.
Note that for~$c = -2$ the critical value of~$f_{-2}$ is equal to~$- \beta(-2)$ and hence it is in~$P_{-2, n}(- \beta(-2))$.
Since~$P_{c, n}(- \beta(c))$ depends continuously with~$c$ on~$\cP_n(-2)$ (Lemma~\ref{lem:hm}) and since~$\cP_n(-2)$ is connected by definition, it follows that for each~$c$ in~$\cP_n(-2)$ the critical value~$c$ of~$f_c$ is in~$P_{c, n}(- \beta(c))$.
\end{proof}

\subsection{The uniformly expanding Cantor set}
\label{ss:expanding Cantor set}
Let~$c$ be a parameter in~$\cP_3(-2)$.
In this subsection we study the maximal invariant set~$\Lambda_c$ of~$f_c^3$ in~$P_{c, 1}(0)$.

Note first that the critical value~$c$ of~$f_c$ is in~$P_{c, 3}(- \beta(c))$ (part~$2$ of Lemma~\ref{lem:auxiliary para-puzzle pieces}) and hence in~$P_{c, 2}(- \beta(c))$.
Since~$\alpha_1(c)$ is in the boundary of~$P_{c, 2}(- \beta(c))$ and since the only external rays that land at this point are~$R_c(5/12)$ and~$R_c(7/12)$ (Lemma~\ref{lem:puzzle pieces}), by Lemma~\ref{lem:0} the external rays~$R_c(7/24)$ and~$R_c(17/24)$ land at the same point, denoted by~$\gamma(c)$, and these are the only external rays that lands at this point.
Similarly, the external rays~$R_c(5/24)$ and~$R_c(19/24)$ land at the same point, denoted by~$\tgamma(c)$, and these are the only external rays that land at~$\tgamma(c)$, see Figure~\ref{f:gamma}.
Note that if in addition~$c$ is real, then the points~$\alpha(c)$ and~$\alpha_1(c)$ are both real (\S\ref{ss:landing to central}); together with the fact that~$c$ is in~$P_{c, 2}(- \beta(c))$, this implies~$c < \alpha_1(c) < \alpha(c)$ (\emph{cf}., part~$2$ of Lemma~\ref{lem:puzzle pieces}).
It follows that the points~$\gamma(c)$ and~$\tgamma(c)$ are both real and that the set~$P_{c, 3}(0) = f_c^{-1}(P_{c, 2}(- \beta(c)))$ satisfies
$$ P_{c, 3}(0) \cap \R = (\gamma(c), \tgamma(c)). $$
\begin{figure}[htb]
\begin{center}
\includegraphics[width = 3.5in]{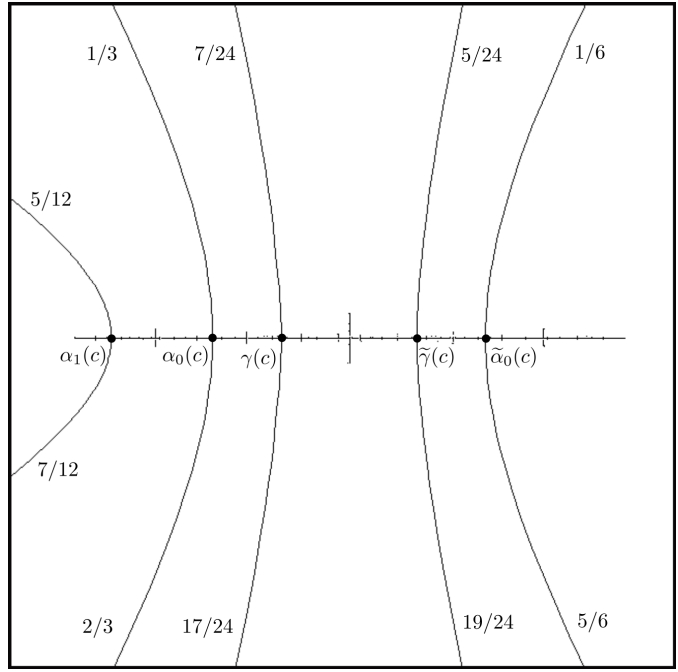}
\end{center}
\caption{Rays landing at~$\alpha_0(c)$, $\talpha_0(c)$, $\alpha_1(c)$, $\gamma(c)$, and~$\tgamma(c)$.}
\label{f:gamma}
\end{figure}

\begin{lemm}
\label{lem:Markov partition}
Let~$c$ be a parameter in~$\cP_3(-2)$.
Then there are precisely~$2$ connected components of~$f_c^{-3}(P_{c, 1}(0))$ contained in~$P_{c, 1}(0)$: One containing~$\gamma(c)$ in its closure, denoted by~$Y_c$, and another one containing~$\tgamma(c)$ in its closure, denoted by~$\tY_c$, see Figure~\ref{f:Y}; the map~$f_c^3$ maps each of the sets~$Y_c$ and~$\tY_c$ biholomorphically to~$P_{c, 1}(0)$.
\begin{figure}[htb]
\begin{center}
\includegraphics[width = 3.5in]{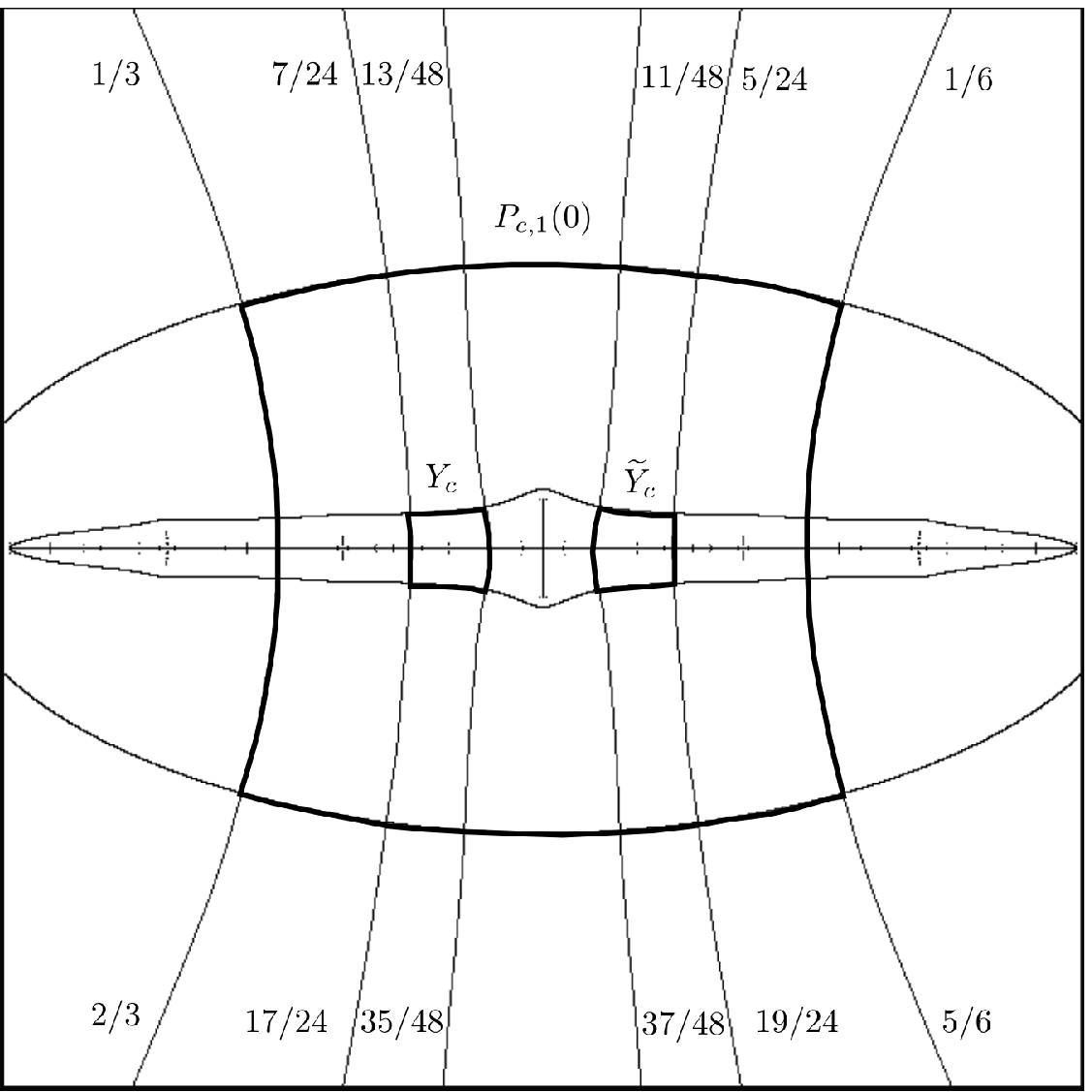}
\end{center}
\caption{The puzzle pieces~$P_{c, 1}(0)$, $Y_c$, and~$\tY_c$.}
\label{f:Y}
\end{figure}
Moreover, the closures of~$Y_c$ and of~$\tY_c$ are disjoint and contained in~$P_{c, 1}(0)$ and the set $Y_c \cup \tY_c$ is contained in~$P_{c, 3}(0)$ and it is disjoint from~$P_{c, 4}(0)$.
Finally, if~$c$ is real, then each of the sets~$Y_c$ and~$\tY_c$ is invariant by complex conjugation and intersects~$\R$.
\end{lemm}
\begin{proof}
We prove first
\begin{equation}
\label{eq:maximal invariant}
f_c^{-3}(P_{c, 1}(0)) \cap P_{c, 1}(0) = f_c^{-1}(V_{c, 2}).
\end{equation}
First notice that, since~$f_c^2$ maps~$V_{c, 2}$ biholomorphically to~$P_{c, 1}(0)$, the set~$f_c^{-1}(V_{c, 2})$ is contained in~$f_c^{-3}(P_{c, 1}(0))$.
On the other hand, the set~$f_c(P_{c, 1}(0)) = P_{c, 0}(- \beta(c))$ contains~$V_{c, 2}$, so~$f_c^{-1}(V_{c, 2})$ is contained in~$P_{c, 1}(0)$.
This proves that the set in the right hand side of~\eqref{eq:maximal invariant} is contained in the set in the left hand side.
To prove the reverse inclusion, let~$z$ be a point in~$P_{c, 1}(0)$ such that~$f_c^3(z)$ is in~$P_{c, 1}(0)$.
Then~$z$ is in a puzzle piece of depth~$4$ and~$f_c(z)$ is in a puzzle piece of depth~$3$ contained in~$P_{c, 1}(-\beta(c))$.
This implies~$f_c^2(z)$ is in~$P_{c, 0}(\beta(c))$.
On the other hand, $f_c^2(z)$ is in~$f_c^{-1}(P_{c, 1}(0)) = V_{c, 1} \cup \tV_{c, 1}$ and~$V_{c, 1}$ is contained in~$P_{c, 1}(- \beta(c)) \subset P_{c, 0}(- \beta(c))$ (part~$1$ of Lemma~\ref{lem:puzzle pieces}), so we conclude that~$f_c^2(z)$ is in~$\tV_{c, 1}$ and hence that~$f_c(z)$ is in~$f_c^{-1}(\tV_{c, 1}) = V_{c, 2} \cup \tV_{c, 2}$.
Since~$f_c(z)$ is in~$P_{c, 1}(- \beta(c))$ and~$\tV_{c, 2}$ is contained in~$P_{c, 2}(\beta(c)) \subset P_{c, 1}(\beta(c))$ (part~$1$ of Lemma~\ref{lem:puzzle pieces}), we conclude that~$f_c(z)$ is in~$V_{c, 2}$ and hence that~$z$ is in~$f_c^{-1}(V_{c, 2})$.
This completes the proof of~\eqref{eq:maximal invariant}.

To prove the assertions of the lemma, note that by part~$2$ of Lemma~\ref{lem:auxiliary para-puzzle pieces} the critical value~$c$ of~$f_c$ is in~$P_{c, 3}(- \beta(c))$, so it is not in the closure of~$V_{c, 2}$.
This implies that~$f_c^{-1}(V_{c, 2})$ has~$2$ connected components whose closures are disjoint.
On the other hand, $V_{c, 2}$ contains~$\alpha_1(c)$ in its closure (\emph{cf}., parts~$1$ and~$2$ of Lemma~\ref{lem:puzzle pieces}), so one of the connected components of~$f_c^{-1}(V_{c, 2})$ contains~$\gamma(c)$ in its closure and the other one contains~$\tgamma(c)$ in its closure; denote them by~$Y_c$ and~$\tY_c$, respectively.
It follows that~$f_c^3$ maps each of the sets~$Y_c$ and~$\tY_c$ biholomorphically to~$P_{c, 1}(0)$.
From the fact that~$V_{c, 2}$ is contained in~$P_{c, 2}(-\beta(c))$ and that the closure of this last set is contained in~$P_{c, 1}(- \beta(c))$ (part~$1$ of Lemma~\ref{lem:puzzle pieces} with~$n = 1$ and~$n = 2$), it follows that closures of~$Y_c$ and~$\tY_c$ are both contained in~$P_{c, 2}(0) = f_c^{-1}(P_{c, 1}(- \beta(c)))$.
Note also that~$V_{c, 2}$ is contained in~$P_{c, 2}(- \beta(c))$ and it is disjoint from~$P_{c, 3}(- \beta(c))$ (part~$1$ of Lemma~\ref{lem:puzzle pieces}), so~$Y_c \cup \tY_c$ is contained in~$P_{c, 3}(0)$ and it is disjoint from~$P_{c, 4}(0)$.
To prove the last statement of the lemma, suppose~$c$ is real.
Then~$f_c$ and~$\alpha_1(c)$ are real and~$V_{c, 2}$ is invariant by complex conjugation (\S\ref{ss:landing to central}).
Since~$c$ is in~$P_{c, 3}(- \beta(c))$ we also have~$c < \alpha_1(c)$.
It follows that each of the sets~$Y_c$ and~$\tY_c$ is invariant by complex conjugation and intersects~$\R$.
This completes the proof of the lemma.
\end{proof}

For a parameter~$c$ in~$\cP_3(-3)$ define
\[
\begin{array}{cccl}
  g_c & : Y_c \cup \tY_c & \to & P_{c,1}(0)\\
      &    z & \mapsto & g_c(z) \= f_c^{3}(z).
 \end{array}
\]
Lemma~\ref{lem:Markov partition} implies that~$g_c$ maps each of the sets~$Y_c$ and~$\tY_c$ biholomorphically to~$P_{c, 1}(0)$ and that
$$ \Lambda_c = \bigcap_{n\in \N} g_c^{-n}(\cl{P_{c,1}(0)}). $$
In particular, $\Lambda_c$ is contained in~$Y_c \cup \tY_c$.
So Lemma~\ref{lem:Markov partition} implies that~$\Lambda_c$ is contained in~$P_{c, 3}(0)$ and that it is disjoint from~$P_{c, 4}(0)$.
Moreover, Lemma~\ref{lem:Markov partition} also implies that~$g_c$ is a Markov map, so~$\Lambda_c$ is a Cantor set and $g_c$ is uniformly expanding on~$\Lambda_c$, see for instance~\cite{dFadMe08}.
In particular, $g_c$ has a unique fixed point in~$Y_c$ and a unique fixed point in~$\tY_c$.
Finally, note that if~$c$ is real, then~$g_c$ is real and~$\Lambda_c$ is contained in~$\R$.

\subsection{Proof of Proposition~\ref{p:ps}}
\label{ss:para-puzzle nest shrinks}
\begin{lemm}
\label{l:landing to central derivative}
There is a constant~$\Delta_1 > 1$ such that for each parameter~$c$ in~$\cP_2(-2)$ the following properties hold for each integer~$k \ge 2$: We have
$$ \Delta_1^{-1} |Df_c(\beta(c))|^{-k}
\le
\diam(P_{c, k}(- \beta(c)))
\le
\Delta_1 |Df_c(\beta(c))|^{-k} $$
and for each point~$y$ in~$P_{c, k}(- \beta(c))$ or in~$P_{c, k}(\beta(c))$ we have
$$ \Delta_1^{-1} |Df_c(\beta(c))|^k \le |Df_c^k(y)|
\le
\Delta_1 |Df_c(\beta(c))|^k. $$
\end{lemm}
\begin{proof}
Since~$P_{c, 1}(\beta(c))$ depends continuously with~$c$ on~$\cP_0(-2)$ (\emph{cf}., Lemma~\ref{lem:hm}) and since~$\cP_0(-2)$ contains the closure of~$\cP_2(-2)$ (part~$1$ of Lemma~\ref{lem:auxiliary para-puzzle pieces}), we have
\begin{align*}
\Xi_1
& \=
\sup_{c \in \cP_2(-2)} \sup_{z \in P_{c, 1}(\beta(c))} |Df_c(z)|
< 
+ \infty,
\\
\Xi_2
& \=
\inf_{c \in \cP_2(-2)} \inf_{z \in P_{c, 1}(\beta(c))} |Df_c(z)|
>
0,
\\
\Xi_3
& \=
\sup_{c \in \cP_2(-2)} \diam(P_{c, 1}(\beta(c)))
<
+ \infty,
\end{align*}
and
$$ \Xi_4
\=
\inf_{c \in \cP_2(-2)} \diam(P_{c, 1}(\beta(c)))
>
0. $$
On the other hand, since for each~$c$ in~$\cP_0(-2)$ the set~ $P_{c,
0}(\beta(c))$ 
contains the closure of~ $P_{c, 1}(\beta(c))$ (\emph{cf}., \S\ref{ss:landing to central}), we have
$$ \Xi_5
\=
\inf_{c \in \cP_2(-2)} \modulus (P_{c, 0}(\beta(c)) \setminus \cl{P_{c, 1}(\beta(c))})
>
0. $$
Let~$\Delta > 1$ be the constant given by Koebe Distortion Theorem with~$A = \Xi_5$.

Let~$c$ be a parameter in~$\cP_2(-2)$ and let~$k \ge 2$ be an integer.
Since~$f_c^{k - 1}$ maps each of the sets~$P_{c, k - 1}(\beta(c))$ and~$P_{c, k - 1}(- \beta(c))$ biholomorphically to~$P_{c, 0}(\beta(c))$, the distortion of~$f_c^{k - 1}$ on~$P_{c, k}(\beta(c))$ is bounded by~$\Delta$.
So for each~$y$ in~$P_{c, k}(- \beta(c))$ or in~$P_{c, k}(\beta(c))$ we have
$$ \Delta^{-1} |Df_c(\beta(c))|^{k - 1}
\le
|Df_c^{k - 1}(y)|
\le
\Delta |Df_c(\beta(c))|^{k - 1}. $$
This implies the first assertion of the lemma with $\Delta_1 = \Delta \max \{ \Xi_1 \Xi_3, \Xi_2^{-1} \Xi_4^{-1} \}$ and second with~$\Delta_1 = \Delta \Xi_1 \Xi_2^{-1}$.
\end{proof}

\begin{proof}[Proof of Proposition~\ref{p:ps}]
By the monotonicity of the kneading invariant, the set~$\cK_n$ is contained in~$(-2, \halpha_{n - 1})$, see~\cite[Theorem~$13.1$]{MilThu88}.
Combined with Lemma~\ref{lem:para-puzzle pieces} this implies that~$\cK_n$ is contained in~$\cP_n(-2)$.
Since~$\halpha_1 = - 3/4$, we also have~$\cK_n \subset (-2, -3/4)$.
To prove that~$\cK_n$ is compact, just observe that from the definitions we have
$$ \cK_n = \{ c \in [-2, \halpha_{n - 1}] \mid f_c^n(c) \in \Lambda_c \}. $$
For a given~$\underline{x}$ in~$\{0, 1 \}^{\N_0}$ the existence and uniqueness of~$c$ in~$\cK_n$ such that~$\iota(c) = \underline{x}$ is a direct consequence of general results of Milnor and Thurston and of Yoccoz, see for example~\cite{MilThu88,dMevSt93} and~\cite{Hub93}.

To prove the last statement of the proposition we show that~$\diam(\cP_n(-2)) \to 0$ as~$n \to + \infty$.
To do this, let~$\Delta_1 > 1$ be given by Lemma~\ref{l:landing to central derivative}, put
$$ \Xi \= \inf_{c \in \cP_2(-2)} |Df_c(\beta(c))| > 1, $$
and let~$\tau : \cP_0(-2) \to \C$ be the holomorphic function defined by~$\tau(c) \= c + \beta(c)$.
A direct computation shows that~$c = -2$ is the only zero of~$\tau$ and that~$\tau'(-2) \neq 0$.
Since the closure of~$\cP_2(-2)$ is contained in~$\cP_0(-2)$ (part~$1$ of Lemma~\ref{lem:auxiliary para-puzzle pieces}), there is a constant $C > 0$  such that for every~$c$ in~$\cP_2(-2)$ we have
\begin{equation}
\label{ine1}
|c-(-2)|
\le
C \frac{|\tau(c)|}{|\tau'(-2)|}.
\end{equation}
Let~$n \ge 2$ be an integer and~$c$ a parameter in~$\cP_n(-2)$.
By part~$2$ of Lemma~\ref{lem:auxiliary para-puzzle pieces} we have $c \in P_{c,n}(-\beta(c))$.
So by Lemma~\ref{l:landing to central derivative} with~$k = n$ and the definition of~$\Xi$ we have,
\[
| \tau(c) |
=
|c-(-\beta(c))|
\le
\Delta_1  |Df_c(\beta(c))|^{-n}
\le
\Delta_1 \Xi^{-n}.
\]
Combining this inequality with~\eqref{ine1}, we conclude that  
$\diam \cP_n(-2)\to 0$ as $n \to +\infty$.
This completes the proof of the proposition.
\end{proof}

\section{Reduced statement}
\label{s:reduced statement}
The purpose of this section is to state a sufficient criterion for a quadratic map corresponding to a parameter in~$\bigcup_{n = 3}^{+ \infty} \cK_n$ to have a low-temperature phase transition (Proposition~\ref{p:first-order phase transition}).
The rest of this section is devoted to prove the Main Theorem using this criterion.
The proof of Proposition~\ref{p:first-order phase transition} occupies~\S\S\ref{s:landing derivatives},~\ref{s:induced map} and~\ref{s:estimating pressure}.

Recall that for a real parameter~$c$,
$$ \chicrit = \liminf_{m \to + \infty} \frac{1}{m} \log |Df_c^m(c)|. $$
\begin{propalph}
\label{p:first-order phase transition}
There is~$n_0 \ge 3$ and a constant~$C_0 > 1$ such that for every integer~$n \ge n_0$ and every parameter~$c$ in~$\cK_n$ the following property holds.
Suppose that for every $t>0$ sufficiently large the sum
$$ \sum_{k =0}^{ + \infty} \exp((n + 3k)t \chicrit / 2) |Df_c^{n + 3k}(c)|^{-t / 2}$$
is less than or equal to~$C_0^{-t}$ and that for some~$t_0 \ge 3$ the sum above with~$t = t_0$ is finite and greater than or equal to~$C_0^{t_0}$.
Then there is~$t_* > t_0$ such that~$f_c|_{I_c}$ (resp.~$f_c|_{J_c}$) has a low-temperature phase transition at~$t = t_*$.
If in addition the sum
$$ \sum_{k = 0}^{+ \infty} k \cdot \exp((n + 3k)t_* \chicrit/2) |Df_c^{n + 3k}(c)|^{-t_*/2} $$
is finite, then~$P_c^{\R}$ (resp.~$P_c^{\C}$) is not differentiable at~$t
= t_*$ and there is a unique equilibrium state of~$f_c|_{I_c}$ (resp.~$f_c|_{J_c}$) for
the potential~$- t_* \log |Df_c|$.
Furthermore, this measure is ergodic, mixing, and its measure-theoretic entropy is strictly positive.
\end{propalph}

After making a uniform distortion bound in~\S\ref{ss:uniform distortion bound}, we give the proof of the Main Theorem in~\S\ref{ss:proof of Main Theorem}. 

\subsection{Uniform distortion bound}
\label{ss:uniform distortion bound}
In this subsection we prove a uniform distortion bound, stated as Lemma~\ref{l:distortion to central 0} below.
We start with some preparatory lemmas.
Recall that for a parameter~$c$ in~$\cP_2(-2)$ the external rays~$R_c(7/24)$ and~$R_c(17/24)$ land at the point~$\gamma(c)$ in~$P_{c,1}(0)$, see~\S\ref{ss:expanding Cantor set}.

\begin{lemm}
\label{l:pleasant couple}
For every parameter~$c$ in~$\cP_2(-2)$ the following properties hold.
\begin{enumerate}
\item[1.]
The open disk~$\hU_c$ containing~$-\beta(c)$ that is bounded by the
equipotential~$2$ and by
\begin{equation}
\label{eq:pleasant cut}
R_c(7/24) \cup \{ \gamma(c) \} \cup R_c(17/24),  
\end{equation}
contains the closure of~$P_{c,0}(- \beta(c))$.
\item[2.]
The open set $\hW_c \= f_c^{-1}(\hU_c)$ contains the closure of~$P_{c, 1}(0)$
and it depends continuously with~$c$ on $\cP_{3}(-2)$.
\end{enumerate}
\end{lemm}
\begin{proof}
\

\partn{1}
Since the puzzle piece~$P_{c, 0}(-\beta(c))$ is bounded by the equipotential~$1$ and by~$R_c(1/3) \cup \{ \alpha(c) \} \cup R_c(2/3)$ (Theorem~\ref{theo:wake 1/2} and \S\ref{ss:landing to central}) and since $7/24 < 1/3 < 2/3 < 17/24$, we deduce that~$\hU_c$ contains the closure of~$P_{c,0}(- \beta(c))$ .

\partn{2}
That~$\hW_c$ contains the closure of~$P_{c, 1}(0) = f_c^{-1}(P_{c, 0}(- \beta(c)))$ is a direct consequence of part~$1$.
To show that~$\hW_c$ depends continuously with~$c$ on~$\cP_3(-2)$, it is enough to show that~$\partial \hW_c$ depends continuously with~$c$ on~$\cP_3(-2)$.
This last assertion follows directly from Lemma~\ref{lem:hm}.
\end{proof}

For the following lemma, see Figure~\ref{f:U}.
\begin{figure}[htb]
\begin{center}
\includegraphics[width = 4.5in]{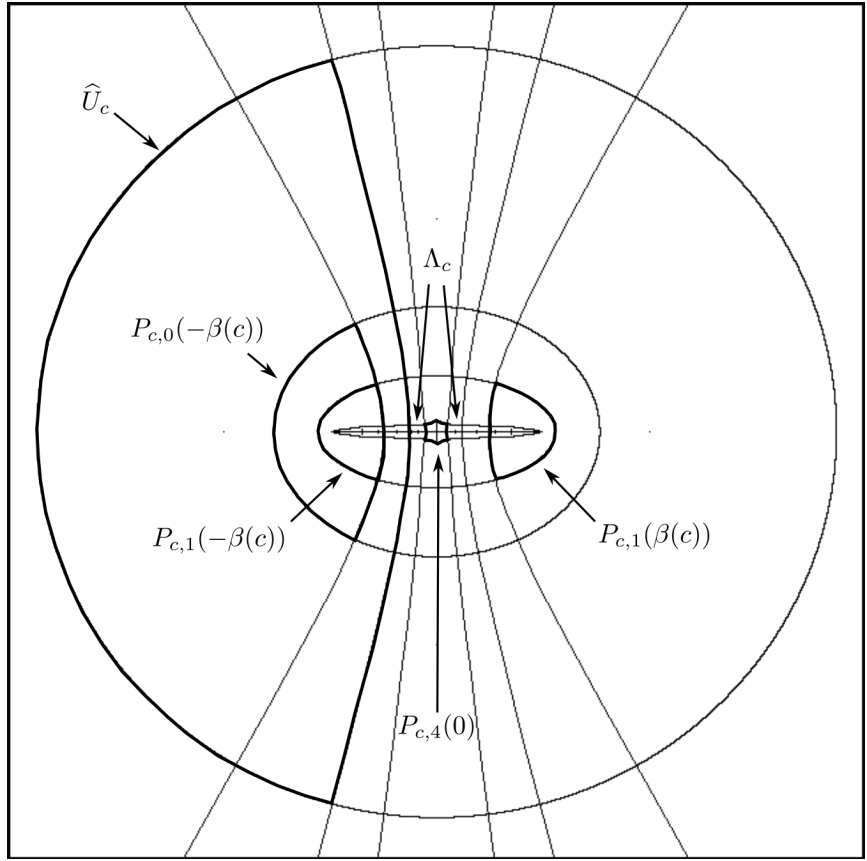}
\end{center}
\caption{The sets~$\hU_c$ and~$\Lambda_c$, and the puzzle pieces~$P_{c, 0}( - \beta(c))$, $P_{c, 1}(-\beta(c))$, $P_{c, 1}(\beta(c))$, and~$P_{c, 4}(0)$.}
\label{f:U}
\end{figure}

\begin{lemm}
\label{l:off postcritical}
Let~$n \ge 3$ be an integer and let~$c$ be a parameter in~$\cK_n$.
Then for every integer~$j \ge 1$ the point~$f_c^j(0)$ is contained in
\begin{equation}
\label{e:post-critical approximation}
P_{c, 1}(- \beta(c)) \cup \Lambda_c \cup P_{c, 1} (\beta(c)).
\end{equation}
Moreover, this set is disjoint from~$\hU_c \setminus P_{c,0}(- \beta(c))$ and from~$P_{c, 4}(0)$.
\end{lemm}
\begin{proof}
By part~$2$ of Lemma~\ref{lem:auxiliary para-puzzle pieces}, the critical value~$c$ of~$f_c$ is in~$P_{c,n}(-\beta(c))$.
Thus for each~$j$ in~$\{1, \ldots, n -1 \}$ the point~$f_c^j(c)$ belongs to~$P_{c, n - j}(\beta(c)) \subset P_{c, 1}(\beta(c))$.
Using the hypothesis that~$c$ is in~$\cK_n$, we conclude that for every integer~$k \ge 0$ the point~$f_c^{n + 3k}(c)$ belongs to~$\Lambda_c$,
that~$f_c^{n + 3k + 1}(c)$ belongs to~$V_{c, 2} \subset P_{c, 1}(- \beta(c))$
and that~$f_c^{n + 3k + 2}(c)$ belongs to~$\tV_{c, 1} \subset P_{c, 1}(- \beta(c))$.
This proves the first part of the lemma.

To prove the last assertion of the lemma, note that~$P_{c, 4}(0)$ is disjoint from~$\Lambda_c$ (\S\ref{ss:expanding Cantor set}).
On the other hand, $P_{c, 4}(0)$ is contained in~$P_{c, 1}(0)$ and it is therefore disjoint from~$P_{c, 1}(-\beta(c)) \cup P_{c, 1}(\beta(c))$.
It remains to prove that~\eqref{e:post-critical approximation} is disjoint from~$\hU_c \setminus P_{c,0}(- \beta(c))$.
This last set is disjoint from~$P_{c, 1}(-\beta(c))$.
To complete the proof, observe that the set~\eqref{eq:pleasant cut} separates~$\C$ into~$2$ connected components: One containing~$- \beta(c)$, denoted by~$H$, and another one containing~$\beta(c)$, denoted by~$\tH$.
Clearly~$\hU_c$ is contained in~$H$.
On the other hand, part~$2$ of Lemma~\ref{lem:puzzle pieces} implies that~$P_{c, 1}(\beta(c))$ is contained in~$\tH$.
Finally, note that~$\Lambda_c$ is contained in~$P_{c, 3}(0)$ (\S\ref{ss:expanding Cantor set}) and that this last set is contained in~$\tH$, see the beginning of~\S\ref{ss:expanding Cantor set}.
This shows that~$\Lambda_c$ and~$P_{c, 1}(\beta(c))$ are both disjoint from~$\hU_c$, and hence from~$\hU_c \setminus P_{c, 0}(-\beta(c))$.
This completes the proof of the lemma.
\end{proof}

\begin{lemm}[Uniform distortion bound]
\label{l:distortion to central 0}
There is~$\Delta_2 > 1$ such that for each integer~$n \ge 4$ and each
parameter~$c$ in~$\cK_n$ the following properties hold:
For each integer~$m \ge 1$ and each connected component~$W$ of~$f_c^{-m}(P_{c, 1}(0))$ on which~$f_c^m$ is univalent, $f_c^m$ maps a neighborhood of~$W$ biholomorphically to~$\hW_c$ and the distortion of this map on~$W$ is bounded by~$\Delta_2$.
\end{lemm}
\begin{proof}
Recall that for each parameter~$c$ in~$\cP_3(-2)$ the set~$\hW_c$ contains the
closure of~$P_{c, 1}(0)$ and that these sets depend continuously with~$c$
on~$\cP_3(-2)$ (\emph{cf}., part~$2$ of Lemma~\ref{l:pleasant couple} and
Lemma~\ref{lem:hm}).
As the closure of~$\cP_4(-2)$ is contained in~$\cP_3(-2)$
(part~$1$ of Lemma~\ref{lem:auxiliary para-puzzle pieces}), we have
$$ A
\=
\inf_{c \in \cP_4(-2)} \modulus (\hW_c \setminus \cl{P_{c, 1}(0)})
>
0. $$
Then the desired assertion follows from Lemma~\ref{l:off postcritical}
and  Koebe Distortion Theorem for this choice
of the constant~$A$.
\end{proof}

\subsection{Proof of Main Theorem assuming Proposition~\ref{p:first-order phase transition}}
\label{ss:proof of Main Theorem}
The following elementary lemma describes the itinerary of the postcritical orbit, for the parameter~$c$ for which we show there is a low-temperature phase transition.
\begin{lemm}
\label{l:phased itinerary}
Let~$N \ge 1$ and~$\ell_0 \ge 1$ be given integers satisfying~$2 \ell_0 \ge N$.
Define~$(a_k)_{k = 0}^{+ \infty}$ as the sequence in~$\{0, 1\}^{\N_0}$ such that for~$k$ in~$\N_0$ we have~$a_k = 0$ if and only if there is an integer~$\ell \ge \ell_0$ such that
$$ \ell^2 \le k \le \ell^2 + N - 1. $$
Moreover, let~$N : \N \to \N_0$ be the function defined for~$k \ge 1$ by
$$ N(k) \= \# \{ j \in \{0, \ldots, k - 1 \} \mid a_j = 0 \} $$
and let~$B : \N \to \N$ be the function defined by $B(1) = 1$ and for~$k \ge 2$ by
$$ B(k) \= 1 + \# \{ j \in \{ 0, \ldots, k - 2 \} \mid a_j \neq a_{j + 1} \}. $$
Then for every~$k$ in~$\{1, \ldots, \ell_0^2 \}$ we have~$N(k) = 0$ and~$B(k) = 1$ and for every~$k \ge \ell_0^2 + 1$ we have
\begin{equation}
  \label{e:approximating N}
B(k) \le 2(\sqrt{k} - \ell_0) + 3
\text{ and }
N \cdot (\sqrt{k} - \ell_0) \le N(k) \le N \sqrt{k}.
\end{equation}
\end{lemm}
\begin{proof}
The assertions for~$k$ in~$\{1, \ldots, \ell_0^2 \}$ and the upper bound of~$B(k)$ are straight forward consequences of the definitions.
Let~$k \ge \ell_0^2 + 1$ be a given integer.
If there is an integer $\ell \ge \ell_0$ such that $\ell^2 \le k - 1 \le \ell^2 + N - 1$, then
\[
N(k) =  N \cdot (\ell - \ell_0) + k - \ell^2
\]
and therefore
\begin{align*}
N \cdot (\sqrt{k}-\ell_0) + (\sqrt{k}-\ell)(\sqrt{k} + \ell - N)
& =
N(k)
\\ & \le
N\sqrt{k} + N \cdot (\ell + 1 - \sqrt{k} -\ell_0).
\end{align*}
Using $N\le 2\ell$ and $\ell + 1 - \sqrt{k} \le 1$, we obtain the estimates for~$N(k)$ in~\eqref{e:approximating N}.
Suppose there is an integer~$\ell \ge \ell_0$ such that
$$ \ell^2 + N \le  k - 1 \le (\ell+1)^2 - 1. $$
Then $N(k) =  N \cdot (\ell - \ell_0 + 1)$ and we also get the bounds for~$N(k)$ in~\eqref{e:approximating N}.
\end{proof}

\begin{proof}[Proof of the Main Theorem]
Let $n_0 \ge 3$ and $C_0 > 1$ be given by Proposition~\ref{p:first-order phase transition} and let $\Delta_1 > 1$ and~$\Delta_2 > 1$ be given by Lemmas~\ref{l:landing to central derivative} and~\ref{l:distortion to central 0}, respectively.

For a given parameter~$c$ in~$\cP_{3}(-2)$ denote by~$p(c)$ the unique fixed point of~$g_c = f_c^3|_{Y_c \cup \tY_c}$  in~$Y_c$ and by~$\wtp(c)$ the unique fixed point of~$g_c$ in~$\tY_c$, see~\S\ref{ss:expanding Cantor set}.
Each of the functions
$$ p : \cP_3(-2) \to \C
\text{ and }
\wtp : \cP_3(-2) \to \C $$
so defined is holomorphic and real.
By Lemma~\ref{l:period 3 transversality} in Appendix~\ref{s:appendix} there is~$\delta > 0$ such that for each parameter~$c$ in the interval~$(-2, -2 + \delta)$ we have
$$ \eta_c \= \frac{|Dg_c(p(c))|}{|Dg_c(\wtp(c))|} > 1. $$
Since for~$c = - 2$ we have
$$ |Dg_{-2}(\wtp(-2))|^{1/3} = 2
\text{ and }
|Df_{-2}(\beta(-2))| = 4, $$
taking~$\delta > 0$ smaller if necessary we assume that for each~$c$ in~$(-2, -2 + \delta)$ we have
\begin{equation}
\label{e:first entrance derivative}
2/3
>
|Dg_c(\wtp(c))|^{1/3} / |f_c(\beta(c))|
>
1/3.
\end{equation}
By Proposition~\ref{p:ps} there is~$n_1 \ge 3$ such that for each integer~$n \ge n_1$ the set~$\cK_n$ is contained in~$(-2, -2 + \delta)$.

Fix a sufficiently large integer~$n \ge \max \{n_0, n_1 \}$ such that
\begin{equation}
\label{e:bound C_1}
\Delta_1^{1/2} \Delta_2^{3/2} (2/3)^{n/2}
<
C_0^{-1}/2.
\end{equation}
Since~$\cK_n$ is compact, we have
$$ \eta \= \inf \{ \eta_c \mid c \in \cK_n \} > 1. $$
Let~$N \ge 1$ be sufficiently large so that $\Delta_2^2 \eta^{- N} < 1$ and 
let~$\ell_0 \ge 1$ be a sufficiently large integer so that 
$$ 2 \ell_0 \ge N 
\text{ and }
\ell_0^2 > C_0^3 (\Delta_1 \Delta_2 3^n)^{1/2}. $$
By Proposition~\ref{p:ps} there is a unique parameter~$c_0$ in~$\cK_n$ such that~$\iota(c_0)$ is given by the sequence~$(a_k)_{k = 0}^{+ \infty}$ defined in Lemma~\ref{l:phased itinerary} for these choices of~$N$ and~$\ell_0$.
To prove the Main Theorem we just need to show that the hypotheses of Proposition~\ref{p:first-order phase transition} are satisfied for this choice of~$n$ and for $c = c_0$ and~$t_0 = 3$.

Note for each integer~$k \ge 1$ the number~$N(k)$ is equal to the number of~$0$'s in the sequence~$(a_j)_{j = 0}^{k - 1}$ and that~$B(k)$ is the number of blocks of~$0$'s or~$1$'s in this sequence.
Let~$k$ be an integer satisfying~$k \ge \ell_0^2 + 1$.
Applying Lemma~\ref{l:distortion to central 0} to each block of~$0$'s or~$1$'s in~$( a_j )_{j = 0}^{k - 1}$, we obtain by~\eqref{e:approximating N} and by the definition of~$\eta$,
\begin{align}
\begin{split}
\label{e:strech exponential estiamte}
\Delta_2^{2(\sqrt{k} -\ell_0 )+3} \left(
\frac{|Dg_{c_0}(p(c_0))|}{|Dg_{c_0}(\wtp(c_0))|} \right)^{N \sqrt{k}}
& \ge
\frac{|Dg_{c_0}^k(f_{c_0}^n(c_0))|}{|Dg_{c_0}(\wtp(c_0))|^k}
\\ & \ge
\Delta_2^{- 2(\sqrt{k} - \ell_0) - 3} \left(
\frac{|Dg_{c_0}(p(c_0))|}{|Dg_{c_0}(\wtp(c_0))|} \right)^{N \cdot (\sqrt{k} - \ell_0)}
\\ & \ge
\Delta_2^{- 3} (\Delta_2^2\eta^{-N})^{-(\sqrt{k} - \ell_0)}.
\end{split}
\end{align}
This implies that
\begin{align}
\begin{split}
\label{eq:chicritbase}
\chicritbase
& =
\lim_{m \to + \infty} \frac{1}{m} \log |Df_{c_0}^m(c_0)|
\\ & =
\frac{1}{3} \lim_{k \to + \infty} \frac{1}{k} \log
|Dg_{c_0}^k(f_{c_0}^n(c_0))|
\\ & =
\log |Dg_c(\wtp(c_0))|^{1/3},
\end{split}
\end{align}
and, by Lemma~\ref{l:landing to central derivative} with~$k = n$ and~$y = c_0$ and by~\eqref{e:first entrance derivative} and~\eqref{e:bound C_1}, that for each integer~$k \ge \ell_0^2 + 1$ we have
\begin{equation}
\label{e:individual term 2}
\exp((n + 3k) \chicritbase/2) |Df_{c_0}^{n + 3k}(c_0)|^{-1/2}
\le
(C_0^{-1}/2) (\Delta_2^2 \eta^{- N})^{(\sqrt{k}-\ell_0)/2}.
\end{equation}
This implies that for every~$t > 0$ the sum
$$ \sum_{k = 0}^{+ \infty} k \cdot \exp(t (n + 3k) \chicritbase/2)
|Df_{c_0}^{n + 3k}(c_0)|^{-t/2}$$
is finite and hence that the last hypothesis of Proposition~\ref{p:first-order phase transition} is automatically satisfied when the other ones are.

To prove that the rest of the hypotheses of Proposition~\ref{p:first-order phase transition} are satisfied, observe that by~\eqref{eq:chicritbase}, by Lemma~\ref{l:distortion to central 0}, by Lemma~\ref{l:landing to central derivative} with~$k = n$ and~$y = c_0$ and by~\eqref{e:first entrance derivative} and~\eqref{e:bound C_1}, for each integer~$k$ in~$\{0, 1,\ldots, \ell_0^2 \}$ we have
\begin{align}
\label{e:individual term 1}
\begin{split}
\Delta_1^{-1/2} \Delta_2^{-1/2} (1/3)^{n/2}
& \le
\exp((n + 3k) \chicritbase/2) |Df_{c_0}^{n + 3k}(c_0)|^{-1/2}
\\ & \le 
\Delta_1^{1/2}\Delta_2^{1/2} (2/3)^{n/2}
\\ & <
C_0^{-1}/2.
\end{split}
\end{align}
So, if for each~$t > 0$ we put 
\[
S(t) \= \sum_{k=\ell_0^2 + 1}^{+\infty}
(\Delta_2^2\eta^{-N})^{t(\sqrt{k}-\ell_0)/2},
\]
then by~\eqref{e:individual term 2} we have
\begin{equation*}
 \sum_{k = 0}^{+ \infty}  \exp(t (n + 3k) \chicritbase/2) |Df_{c_0}^{n + 3k}(c_0)|^{-t/2}
<
\frac{C_0^{-t}}{2^t}(\ell_0^2 + 1 + S(t)).
\end{equation*}
By our choice of~$N$ we have~$\Delta_2^2 \eta^{-N} < 1$ and hence~$S(t) \to 0$ as~$t \rightarrow +\infty$.
Together with the fact that~$C_0 > 1$, this proves that the sum above converges to~$0$ as~$t \to + \infty$.
Finally, note that by~\eqref{e:individual term 1} and our choice of~$\ell_0$, the sum above with~$t = 3$ is greater than~$C_0^3$.
This completes the proof that the hypotheses of Proposition~\ref{p:first-order phase transition} are satisfied with~$c = c_0$ and~$t_0 = 3$ and thus completes the proof of the Main Theorem.
\end{proof}

\section{Expansion away from the critical point}
\label{s:landing derivatives}
For an integer~$n \ge 4$ and a parameter~$c$ in~$\cP_n(-2)$ put~$V_c \= P_{c, n + 1}(0)$ and denote by~$D_c'$ the set of all those points~$z$ in~$\C \setminus V_c$ for which there is an integer~$m \ge 1$ such that~$f^{m}(z)$ is in~$V_c$; for such~$z$ denote by~$m_c(z)$ the least integer~$m$ with this property and call it the \emph{first landing time of~$z$ to~$V_c$}.
The \emph{first landing map to~$V_c$} is the map~$L_c : D_c' \to V_c$ defined by~$L_c(z) = f_c^{m_c(z)}(z)$.

The purpose of this section is to prove the following proposition.
\begin{propalph}
\label{p:landing derivatives}
There is a constant~$C_1 > 1$ such that for each~$\varepsilon > 0$ there is~$n_2 \ge 4$ such that the following property holds: For each
integer~$n \ge n_2$, each parameter~$c$ in~$\cK_n$, and each~$z$ in~$L_c^{-1}(V_c)$ we have
$$ |DL_c(z)|
\ge
C_1^{-1} 2^{m_c(z)(1 - \varepsilon)}. $$
\end{propalph}

To prove this proposition we first show that the restriction of~$L_c$ to each connected component of its domain admits a univalent extension onto~$P_{c, 1}(0)$ (Lemma~\ref{l:univalent pull-back property}).
The proof of the Proposition~\ref{p:landing derivatives} is given in~\S\ref{ss:landing derivatives}, after some derivative
estimates stated as Lemmas~\ref{l:quasi tent estimate} and~\ref{l:first return to central derivative}.

\subsection{Univalent pull-back property}
\label{ss:univalent pull-back property}
Note that the domain~$D_c'$ of~$L_c$ is a disjoint union of puzzle pieces, so each connected component of~$D_c'$ is a puzzle piece.
Furthermore, for each connected component~$W$ of~$D_c'$ the first landing time to~$V_c$ of all points in~$W$ is the same; denote the common value by~$m_c(W)$.
So~$L_c$ maps~$W$ biholomorphically to~$V_c$.
\begin{lemm}[Univalent pull-back property]
\label{l:univalent pull-back property}
For every integer~$\ene \ge 0$ and every parameter~$c$ in~$\cP_{\ene}(-2)$, the following property holds.
Let $m \ge 1$ be an integer and~$z$ a point in~$f_c^{-m}(P_{c, 1}(0))$ such that for each~$j$ in~$\{0, \ldots, m - 1 \}$ we have~$f_c^j(z) \not \in P_{c, \ene + 1}(0)$.
Then the puzzle piece~$P$ of depth~$m + 1$ containing~$z$ is such that for every~$j$ in~$\{ 0, \ldots, m - 1 \}$ the set~$f_c^j(P)$ is disjoint from~$P_{c,
\ene + 1}(0)$ and~$f_c^m$ maps~$P$ biholomorphically to~$P_{c, 1}(0)$.
If in addition~$c$ is real, then~$P$ intersects the real line.
\end{lemm}

The proof of this lemma is below, after the following one.
\begin{lemm}
\label{l:landing to central}
Let~$c$ be a parameter in~$\cP_0(-2)$, let~$\ell \ge 1$ be an integer, and let~$z$ be a point in~$f_c^{-\ell}(P_{c, 1}(0))$ such that for each~$j$ in~$\{0, \ldots, \ell - 1 \}$ the point~$f_c^j(z)$ is not in~$P_{c, 1}(0)$.
Then~$z$ is in~$V_{c, \ell}$ or in~$\tV_{c, \ell}$.
\end{lemm}
\begin{proof}
We proceed by induction in~$\ell$.
The case~$\ell = 1$ follows from
$$ f_c^{-1}(P_{c, 1}(0))
=
\phi_c(P_{c, 1}(0)) \cup \tphi_c (P_{c, 1}(0))
=
V_{c, 1} \cup \tV_{c, 1}. $$
Let~$\ell \ge 2$ be an integer and suppose the desired assertion holds with~$\ell$ replaced by~$\ell - 1$.
If~$z$ is as in the lemma, then~$z$ is not in~$P_{c, 1}(0)$ and~$G_c(z) \le 1/ 2^\ell \le 1/2$.
Therefore~$z$ is in either~$P_{c, 1}(-\beta(c))$ or~$P_{c, 1}(\beta(c))$; in both cases~$f_c(z)$ is in~$P_{c, 0}(\beta(c))$.
Applying the induction hypothesis to~$f_c(z)$ we conclude that~$f_c(z)$ is in~$\tV_{c, \ell - 1}$ and therefore that~$z$ is in
$$ f_c^{-1}(\tV_{c, \ell - 1}) = V_{c, \ell} \cup \tV_{c, \ell}. $$
This completes the proof of the induction step and of the lemma.
\end{proof}

\begin{proof}[Proof of Lemma~\ref{l:univalent pull-back property}]
We proceed by induction in~$m$.
Since~$f_c^{-1}(P_{c, 1}(0)) = V_{c, 1} \cup \tV_{c, 1}$, the desired assertions clearly hold for~$m = 1$.
Given an integer~$m \ge 2$, suppose by induction that the desired assertions hold for every integer less than or equal to~$m - 1$.
Given~$z$ as in the statement of the lemma, let~$P$ be the puzzle piece of~$f_c$
of depth~$m + 1$ containing~$z$, so that~$f_c^m(P) = P_{c, 1}(0)$. 

First, we prove that $f_c^m$ maps $P$ biholomorphically to $P_{c,1}(0)$.
Let~$\ell \ge 1$ be the least integer such that~$f_c^\ell(P)$ is contained
in~$P_{c,
1}(0)$; we have~$\ell \le m$.
If~$P$ is not contained in~$P_{c, 1}(0)$, then it is contained in~$V_{c, \ell}$
or~$\tV_{c, \ell}$ (Lemma~\ref{l:landing to central}); in both cases~$P$ is contained in a puzzle piece of
depth~$\ell + 1$ that is mapped biholomorphically to~$P_{c, 1}(0)$ by~$f_c^{\ell}$.
If~$P$ is contained in~$P_{c, 1}(0)$, then~$\ell \ge 2$, $f_c(P)$ is contained in~$V_{c, \ell - 1}$ (Lemma~\ref{l:landing to central}) and hence in~$P_{c, \ell - 1}(- \beta(c))$ (part~$1$ of Lemma~\ref{lem:puzzle pieces}).
Since our hypotheses imply that~$f_c(P)$ is not contained in~$P_{c, \ene}(c)$ and since this last puzzle piece is equal to~$P_{c, \ene}(- \beta(c))$ (part~$2$ of Lemma~\ref{lem:auxiliary para-puzzle pieces}), we conclude that~$P_{c, \ene}(-\beta(c))$ is strictly contained in~$P_{c, \ell - 1}(-\beta(c))$ and therefore that~$\ell - 1 < \ene$.
So the puzzle piece of~$f_c$ of depth~$\ell + 1$ containing~$P$ is mapped
biholomorphically to~$V_{c, \ell - 1}$ by~$f_c$; this proves that in all the cases~$f_c^{\ell}$ maps the puzzle piece of depth~$\ell + 1$ containing~$P$
biholomorphically to~$P_{c, 1}(0)$ and shows the inductive step in the case where~$\ell = m$.
If~$\ell \le m - 1$, then by the induction hypothesis applied to~$m - \ell$
instead of~$m$ and with~$f_c^{\ell}(z)$ instead of~$z$, we conclude that~$f_c^{m
- \ell}$ maps~$f_c^{\ell}(P)$ biholomorphically to~$P_{c, 1}(0)$.
This completes the proof that~$f_c^m$ maps~$P$ biholomorphically to~$P_{c, 1}(0)$.

Now we prove the other assertions of the lemma.
For each~$j$ in~$\{0, \ldots, m - 1\}$ we have~$f_c^j(z) \not \in P_{c, \ene + 1}(0)$.
Let~$P'$ be the puzzle piece of depth~$m$ containing~$z' \= f_c(z)$.
By our induction hypothesis we just need to prove that~$P$ is disjoint
from~$P_{c, \ene + 1}(0)$, and if~$c$ is real, that~$P$ intersects~$\R$.
Suppose~$z$ is not in~$P_{c, 1}(0)$.
Then by Lemma~\ref{l:landing to central} there is an integer~$\ell \ge 1$
such that~$z$ belongs to~$V_{c, \ell}$ or~$\tV_{c, \ell}$.
Then~$m \ge \ell$ and~$P$ is contained in one of these sets; it follows that~$P$
is disjoint from~$P_{c, 1}(0)$ and hence from~$P_{c, \ene + 1}(0)$.
Suppose~$c$ is real.
Then the maps~$\phi_c$ and~$\tphi_c$ are both real and by our induction
hypothesis~$P'$ intersects~$\R$.
Since~$P$ is equal to either~$\phi_c(P')$ or~$\tphi_c(P')$, it follows that~$P$
also intersects~$\R$.
It remains to consider the case where~$z$ belongs to~$P_{c, 1}(0)$.
Since by hypothesis~$z$ is not in~$P_{c, \ene + 1}(0)$, the point~$z'$ is not
in~$P_{c, \ene}(- \beta(c))$.
So there is an integer~$\ell \le m - 1$ in~$\{ 1, \ldots, \ene - 1 \}$ such that~$z'$ is
in~$V_{c, \ell}$ (Lemma~\ref{l:landing to central}).
It follows that~$P'$ is contained in~$V_{c, \ell}$ and that it is therefore
disjoint from~$P_{c, \ene}(- \beta(c))$; this implies that~$P$ is disjoint
from~$P_{c, \ene + 1}(0)$.
If~$c$ is real, then by the induction hypothesis~$P'$ intersects~$\R$.
Since~$P'$ is contained in~$V_{c, \ell}$ and~$\ell \le \ene - 1$, it follows
that~$P' \cap \R$ is contained in~$f_c(\R)$.
This implies that~$P$ intersects~$\R$ and completes the proof of the induction step and of the lemma.
\end{proof}

\subsection{Derivatives estimates}
\label{ss:landing derivatives}
\begin{lemm}
\label{l:quasi tent estimate}
There is a constant~$C_2 > 1$ such that for every~$\varepsilon > 0$ and every integer~$m_1 \ge 1$ there is~$n_3 \ge 3$ such that the following property holds for each integer~$n \ge n_3$ and each parameter~$c$ in~$\cK_n$: For every integer~$m \ge 1$ and every point~$z$ in~$f_c^{-m}(P_{c, 1}(0))$ such that for each~$j$ in~$\{ 0, \ldots, m - 1 \}$ we have~$f_c^j(z) \not \in P_{c, m_1}(0)$, we have
$$ |Df_c^m(z)|
\ge
C_2^{-1} 2^{m (1 - \varepsilon)}. $$
If in addition~$z$ is in~$P_{c, 1}(0)$, then
$$ |Df_c^m(z)|
\le
C_2 2^{m (1 + \varepsilon)}. $$
\end{lemm}
\begin{proof}
Let~$\Delta_2 > 1$ be the constant given by Lemma~\ref{l:distortion to central 0}.
Given a parameter~$c$ in~$[-2, 1/4)$ consider the smooth homeomorphism
$$ \begin{array}{rcl}
h_c : [0, 1] & \to & [- \beta(c), \beta(c)]  \\
\theta & \mapsto & h_c(\theta) \= \beta(c) \cos(\pi \theta);
\end{array} $$
it depends smoothly on~$c$ and when~$c = -2$ we have
$$ \inf_{\substack{x \in [-\beta(- 2), \beta(- 2)],
\\
~y \in [\alpha(- 2), \talpha(-2)]}}
|Dh_{-2}(h_{-2}^{-1}(y))|/|Dh_{-2}(h_{-2}^{-1}(x))| > 0. $$
So there is~$\delta_0 \in (0, 9/4)$ such that
$$ \kappa
\=
\inf_{c \in [-2, -2 + \delta_0]} \inf_{\substack{x \in [-\beta(c), \beta(c)],
 \\ 
y \in [\alpha(c), \talpha(c)]}}
|Dh_c(h_c^{-1}(y))| / |Dh_c(h_c^{-1}(x))|
>
0 $$
and
$$ \hkappa
\=
\inf_{c \in [-2, -2 + \delta_0]} \inf_{\substack{x \in [\alpha(c), \talpha(c)], 
\\
y \in [\alpha(c), \talpha(c)]}}
|Dh_c(h_c^{-1}(y))| / |Dh_c(h_c^{-1}(x))|
< + \infty. $$

Let~$\varepsilon > 0$ and let~$m_1 \ge 1$ be given.
Taking~$m_1$ larger if necessary, we assume~$m_1 \ge 4$.
Since~$P_{c, m_1}(0)$ depends continuously with~$c$ on~$\cP_{m_1 - 1}(-2)$ 
(\emph{cf}., Lemma~\ref{lem:hm}) and since this last set contains the closure
of~$\cP_{m_1}(-2)$ (part~$1$ of Lemma~\ref{lem:auxiliary para-puzzle pieces}), there is~$\tau > 0$
such that for each parameter~$c$ in~$\cP_{m_1}(-2)\cap [-2,1/4)$ we have
$$ [1/2 - \tau, 1/2 + \tau]
\subset
h_c^{-1}(P_{c, m_1}(0)\cap [-\beta(c), \beta(c)]). $$
For each parameter~$c$ in~$[-2, 1/4)$ let~$T_c : [0, 1] \to [0, 1]$ be the map defined by~$T_c = h_c^{-1} \circ f_c \circ h_c$.
When~$c = - 2$ the map~$T_{- 2}$ is the tent map given by $T_{-2}(\theta) = 2 \theta$ on~$[0, 1/2]$ and $T_{-2}(\theta) = 2 - 2 \theta$ on~$[1/2, 1]$.
When~$c$ is not equal to~$-2$, the map~$T_c$ is smooth on~$[0, 1]$.
A direct computation shows that there is~$\delta_1$ in~$(0, \delta_0)$ such that for each parameter~$c$ in~$[-2, -2 + \delta_1]$
and each~$\theta$ in~$[0, 1]$ satisfying~$|\theta - 1/2| \ge \tau$, we have
$$  2^{1 - \varepsilon}\le  |DT_c(\theta)| \le  2^{1 + \varepsilon}. $$
Let~$n_1$ be given by Proposition~\ref{p:ps} with~$\delta = \delta_1$.

Fix an integer~$n \ge \max \{n_1, m_1 \}$ and a parameter~$c$ in~$\cK_n$.
By Proposition~\ref{p:ps} we have $\cK_n \subset (- 2, - 2 + \delta_1)$.
Let $m \ge 1$ be an integer, $z$ a point in~$f_c^{-m}(P_{c, 1}(0))$ such that
for each~$j$ in~$\{0, \ldots, m - 1 \}$ we have~$f_c^j(z) \not \in P_{c, m_1}(0)$
and let~$P$ be the puzzle piece of~$f_c$ of depth~$m + 1$ that contains~$z$.
By Lemma~\ref{l:univalent pull-back property} with~$n$ replaced by~$m_1 - 1$ there is a real point~$x$ in~$P$
and for every~$j$ in~$\{ 0, \ldots, m - 1 \}$ the point~$f_c^j(x)$ of~$f_c^j(P)$ is not in~$P_{c, m_1}(0)$; by our choice of~$\tau$ it follows that~$h_c^{-1}(f_c^j(x))$
is not in~$[1/2 - \tau, 1/2 + \tau]$.
On the other hand, $f_c^m$ maps~$P$ biholomorphically to~$P_{c, 1}(0)$ and by 
Lemma~\ref{l:distortion to central 0} the distortion of~$f_c^m$ on~$P$ is
bounded by~$\Delta_2$.
Since~$x$ is in~$[- \beta(c), \beta(c)]$ and
$$ f_c^m(x)
\in
f_c^m(P) \cap \R
=
P_{c, 1}(0) \cap \R
=
(\alpha(c), \talpha(c)), $$
by the considerations above we have by the definition of~$\kappa$,
\begin{align*}
|Df_c^m(z)|
& \ge
\Delta_2^{-1} |Df_c^m(x)|
\\ & =
\Delta_2^{-1} |Dh_c(h_c^{-1}(f_c^m(x)))| \cdot |DT_c^m (h_c^{-1}(x))| /
|Dh_c(h_c^{-1}(x))|
\\ & \ge 
\Delta_2^{-1} \kappa 2^{m (1 - \varepsilon)}.
\end{align*}
If in addition~$z$ is in~$P_{c, 1}(0)$, then~$x$ belongs to~$(\alpha(c), \talpha(c))$ and by the considerations above we have by the definition of~$\hkappa$,
\begin{align*}
|Df_c^m(z)|
& \le
\Delta_2 |Df_c^m(x)|
\\ & =
\Delta_2 |Dh_c(h_c^{-1}(f_c^m(x)))| \cdot |DT_c^m (h_c^{-1}(x))| /
|Dh_c(h_c^{-1}(x))|
\\ & \le 
\Delta_2 \hkappa^{-1} 2^{m (1 + \varepsilon)}.
\end{align*}
This proves the lemma with~$n_2 = \max \{n_1, m_1 \}$ and~$C_2 = \Delta_2^{-1} \max \{ \kappa^{-1}, \hkappa \}$.
\end{proof}

\begin{lemm}
\label{l:first return to central derivative}
There is~$C_3 > 1$ such that for each integer~$n \ge 4$ and each
parameter~$c$ in~$\cK_n$ the following properties hold for each integer~$q \ge
1$.
\begin{enumerate}
\item[1.]
For each open set~$W$ that is mapped biholomorphically to~$P_{c, 1}(0)$ 
by~$f_c^{q}$ and each~$x$ in~$W$ we have
$$ |Df_c(x)| \ge C_3^{-1} |Df_c^{q - 1} (f_c(x))|^{-1/2}. $$
\item[2.]
If~$q - 1 \neq n$, then for each point~$x$ in~$f_c^{-1}(V_{c, q - 1})$ we have
$$ |Df_c^{q} (x)| \ge C_3^{-1} |Df_c(\beta(c))|^{q/2}. $$
\end{enumerate}
\end{lemm}
\begin{proof}
Let~$\Delta_1 > 1$ and~$\Delta_2 > 1$ be the constants given by Lemmas~\ref{l:landing to central derivative} and~\ref{l:distortion to central 0}, respectively.

Since the sets~$P_{c, 1}(\beta(c))$ and $P_{c, 1}(0)$
are disjoint and depend continuously with~$c$ on~$\cP_0(-2)$
(\emph{cf}., \S\ref{ss:puzzles} and Lemma~\ref{lem:hm}) and since~$\cP_0(-2)$
contains the closure of~$\cP_4(-2)$ (part~$1$ of Lemma~\ref{lem:auxiliary para-puzzle pieces}), we have
$$ \Xi_1
\=
\inf_{c \in \cP_4(-2)} \diam(P_{c, 1}(0))
>
0
\text{ and }
\Xi_2
\=
\sup_{c \in \cP_4(-2)} |Df_c(\beta(c))|
< 
+ \infty. $$
On the other hand, for each~$c$ in~$\cP_3(-2)$ the closure of~$P_{c, 1}(0)$ is contained in~$\hW_c$ and~$\hW_c$ depends continuously with~$c$ on~$\cP_3(-2)$ (part~$2$ Lemma~\ref{l:pleasant couple}); so
$$ \Xi_3
\=
\inf_{c \in \cP_{4}(-2)} \modulus (\hW_c \setminus \cl{P_{c, 1}(0)})
>
0. $$

Let~$n \ge 4$ be a integer and~$c$ a parameter in~$\cK_n$.

\partn{1}
Note that~$f_c^{q}$ maps a neighborhood~$\hW$ of~$W$ biholomorphically to~$\hW_c$ (Lemma~\ref{l:distortion to central 0}).
So if we put~$\hW' \= f_c(\hW)$, then~$c$ is not in~$\hW'$ and~$f_c^{q - 1}$ 
maps~$\hW'$ biholomorphically to~$\hW_c$; in particular we have 
$$ \modulus (\hW' \setminus \cl{f_c(W)})
=
\modulus (\hW_c \setminus \cl{P_{c, 1}(0)})
\ge
\Xi_3. $$
Thus there is a constant~$A_1 > 0$ independent of~$n$, $c$ and~$q$ such that for every~$x$ in~$W$, we have 
\begin{equation*}
|f_c(x) - c|
\ge
\dist(f_c(W), c)
\ge
\dist(f_c(W), \partial \hW')
\ge
A_1 \diam(f_c(W))
\end{equation*}
(\emph{cf}., \cite[Teichm{\"u}ller's module theorem, {\S}II.$1$.$3$]{LehVir73}).
Thus, if we put~$A_2 \= 2(A_1 \Delta_2^{-1} \Xi_1)^{1/2}$, then by Lemma~\ref{l:distortion to central 0} with~$m = q - 1$ and with~$W$ replaced by~$f_c(W)$ we have
$$ |Df_c(x)|
\ge
2 A_1^{1/2} \diam(f_c(W))^{1/2}
\ge
A_2 |Df_c^{q - 1} (f_c(x))|^{- 1/2}. $$
This proves part~$1$ with constant~$C_3 = A_2^{-1}$.

\partn{2}
Since~$f_c(x)$ is in~$V_{c, q - 1}$ and this last set is contained in~$P_{c, q - 1}(-\beta(c))$ (part~$1$ of Lemma~\ref{lem:puzzle pieces}), by Lemma~\ref{l:landing to central derivative} with~$y = f_c(x)$ and~$k = q - 1$, we have
\begin{equation}
\label{e:landing to central derivative}
|Df_c^{q - 1} (f_c(x))| \ge \Delta_1^{-1} |Df_c(\beta(c))|^{q - 1}.
\end{equation}
Our assumption~$q - 1 \neq n$ implies that~$f_c(0) = c$ is not in~$V_{c, q - 1}$, so~$f_c^{q}$ maps the connected component~$W$ of~$f_c^{-1}(V_{c, q - 1})$ containing~$x$ biholomorphically to~$P_{c, 1}(0)$.
So the desired assertion with~$C_3$ replaced by~$C_3 (\Delta_1 \Xi_2)^{1/2}$ follows from~\eqref{e:landing to central derivative} and from part~$1$.
\end{proof}

\begin{proof}[Proof of Proposition~\ref{p:landing derivatives}]
Let~$C_2$ and~$C_3$ be the constants given by Lemmas~\ref{l:quasi tent estimate} and~\ref{l:first return to central derivative}, respectively.
Let~$m_1 \ge 2$ be sufficiently large so that
$$ 2^{(m_1 - 1) \varepsilon/2}
\ge
C_2 C_3 $$
and let~$n_3$ be given by Lemma~\ref{l:quasi tent estimate} for this choice of~$m_1$.
Notice that for~$c = -2$ we have~$Df_{-2}(\beta(-2)) = 4$.
So, in view of Proposition~\ref{p:ps}, we can take~$n_3$ larger if necessary and assume that for each parameter~$c$ in~$\cP_{n_3}(-2)$ we have
$$ |Df_c(\beta(c))|^{1/2}
\ge
2^{1 -\varepsilon /2}. $$

We prove the desired assertion with~$n_2 = n_3$ and~$C_1 = C_2$.
To do this, let~$n \ge n_3$ be an integer, $c$ a parameter in~$\cK_n$, and let~$z$ be a point in~$L_c^{-1}(V_c)$.
If for every~$j$ in~$\{0, \ldots, m_c(z) - 1 \}$ we have~$f_c^j(z) \not \in P_{c, m_1}(0)$, then the desired assertion follows from Lemma~\ref{l:quasi tent estimate} with~$m = m_c(z)$.
So we assume that there is~$\ell$ in~$\{0, \ldots, m_c(z) - 1 \}$ such that~$f_c^\ell(z)$ belongs to~$P_{c, m_1}(0)$.
Let~$k \ge 1$ be the number of all such integers, let~$\ell_1 < \ell_2 < \cdots < \ell_k$ be the increasing sequence of all of these numbers, and put~$\ell_{k + 1} \= m_c(z)$.
Given~$s$ in~$\{ 1, \ldots, k \}$ let~$\ell_s'$ be the least integer~$\ell \ge \ell_s + 1$ such that~$f_c^{\ell}(z)$ is in~$P_{c, 1}(0)$.
Then~$\ell_s' \le \ell_{s + 1}$, $\ell_s' - \ell_s \ge m_1 - 1$, and the point $f_c^{\ell_s + 1}(z)$ belongs to~$V_{c, \ell_s' - \ell_s - 1}$ (Lemma~\ref{l:landing to central}).
By our choice of~$z$, the point~$f_c^{\ell_s}(z)$ does not belong to~$V_c = P_{c, n + 1}(0) = f_c^{-1}(V_{c, n})$, so~$\ell_s' - \ell_s - 1 \neq n$ and by part~$2$ of Lemma~\ref{l:first return to central derivative} with~$q = \ell_s' - \ell_s$ and~$x = f_c^{\ell_s}(z)$ and by our choice of~$n_3$ and~$m_1$ we have
\begin{align}
\label{e:derivative from deep}
\begin{split}
|Df_c^{\ell_s' - \ell_s}(f_c^{\ell_s}(z))|
& \ge
C_3^{-1} |Df_c(\beta(c))|^{(\ell_s' - \ell_s)/2}
\\ & \ge
C_3^{-1} 2^{(\ell_s' - \ell_s)(1 - \varepsilon/2)}
\\ & \ge
C_2 2^{(\ell_s' - \ell_s)(1 - \varepsilon)}.
\end{split}
\end{align}
When~$\ell_s' = \ell_{s + 1}$ we obtain
\begin{equation}
\label{e:derivative from deep to deep}
|Df_c^{\ell_{s + 1} - \ell_s} (f_c^{\ell_s}(z))|
\ge
2^{(\ell_{s + 1} - \ell_s)(1 - \varepsilon)}.
\end{equation}
In the case where~$\ell_s' \le \ell_{s + 1} - 1$, the point~$f_c^{\ell_s'}(z)$ belongs to~$P_{c, 1}(0)$ but not to~$P_{c, m_1}(0)$; so~\eqref{e:derivative from deep} together with Lemma~\ref{l:quasi tent estimate} with~$m = \ell_{s + 1} - \ell_s'$ and with~$z$ replaced by~$f_c^{\ell_s'}(z)$ implies, by our choice of~$n_3$, that
\begin{equation*}
|Df_c^{\ell_{s + 1} - \ell_s}(f_c^{\ell_s}(z))|
\ge
|Df_c^{\ell_{s + 1} - \ell_s'}(f_c^{\ell_s'}(z))| C_2 2^{(\ell_s' - \ell_s)(1 - \varepsilon)}
\ge
2^{(\ell_{s + 1} - \ell_s)(1 - \varepsilon)}.
\end{equation*}
So in all the cases we obtain~\eqref{e:derivative from deep to deep} and therefore
\begin{equation}
\label{e:from deep to full time}
|Df_c^{m_c(z) - \ell_1} (f_c^{\ell_1}(z))|
=
\prod_{s = 1}^{k} |Df_c^{\ell_{s + 1} - \ell_s} (f_c^{\ell_s}(z))|
\ge
2^{(m_c(z) - \ell_1)(1 - \varepsilon)}.
\end{equation}
This proves the desired inequality in the case where~$\ell_1 = 0$.
If~$\ell_1 \ge 1$, then by Lemma~\ref{l:quasi tent estimate} with~$m = \ell_1$ we have
$$ |Df_c^{\ell_1} (z)|
\ge
C_2^{-1} 2^{\ell_1 (1 - \varepsilon)}. $$
Together with~\eqref{e:from deep to full time} this implies the desired inequality and completes the proof of the proposition.
\end{proof}

\section{Induced map}
\label{s:induced map}
In this section, for a parameter~$c$ in~$\cP_4(-2)$ we use the first return map~$F_c$ of~$f_c$ to~$V_c$ to study~$P_c^{\R}$ and~$P_c^{\C}$.
After some basic considerations in~\S\ref{ss:induced map}, we show that~$P_c^{\R}$ and~$P_c^{\C}$ are related to a~$2$ variables pressure function of~$F_c$ through a Bowen type formula, see Proposition~\ref{p:Bowen type formula} in~\S\ref{ss:Bowen type formula} and compare with~\cite{StrUrb03}
and~\cite{PrzRiv11}.
We do this by analyzing the convergence properties of a suitable Poincar{\'e}
series (Lemma~\ref{l:Poincare series}).
In the proof of Proposition~\ref{p:Bowen type formula} we use a lower bound for~$P_c^{\C}$ (Proposition~\ref{p:critical line} in~\S\ref{ss:critical line}) that is used again in the next section.

\subsection{Induced map}
\label{ss:induced map}
Let~$n \ge 4$ be an integer and~$c$ a parameter in~$\cK_n$.
Throughout the rest of this section put $\hV_c \= P_{c, 4}(0)$.
Since the critical value~$c$ of~$f_c$ is in~$P_{c, n}(- \beta(c))$ (part~$2$ of Lemma~\ref{lem:auxiliary para-puzzle pieces}), the closure of~$V_c = P_{c, n + 1}(0) = f_c^{-1}(P_{c, n}(- \beta(c)))$ is contained in~$\hV_c = f_c^{-1}(P_{c, 3}(- \beta(c)))$ (\emph{cf}., part~$1$ of Lemma~\ref{lem:puzzle pieces}).

Let~$D_c$ be the set of all those points~$z$ in~$V_c$ for which there is an integer~$m \ge 1$ such that~$f_c^m(z)$ is in~$V_c$.
For~$z$ in~$D_c$ denote by~$m_c(z)$ the least integer~$m$ with this property and call it the \emph{first return time of~$z$ to~$V_c$}.
The \emph{first return map to~$V_c$} is defined by
$$ \begin{array}{rcl}
F_c : D_c & \to & V_c \\
z & \mapsto & F_c(z) \= f_c^{m_c(z)}(z).
\end{array} $$
It is easy to see that~$D_c$ is a disjoint union of puzzle pieces; so each connected component of~$D_c$ is a puzzle piece.
Note furthermore that in each of these puzzle pieces~$W$, the return time function~$m_c$ is constant; denote the common value of~$m_c$ on~$W$ by~$m_c(W)$.
\begin{lemm}[Uniform bounded distortion]
\label{l:bounded distortion to nice}
There is a constant~$\Delta_3 > 1$ such that for each integer~$n \ge 5$ and each
parameter~$c$ in~$\cK_n$ the following property holds: For every connected component~$W$ of~$D_c$ the map~$F_c|_W$ is univalent and its distortion is bounded by~$\Delta_3$.
Furthermore, the inverse of~$F_c|_W$ admits a univalent extension to~$\hV_c$ taking images in~$V_c$.
In particular, $F_c$ is uniformly expanding with respect to the hyperbolic metric on~$\hV_c$.
\end{lemm}
\begin{proof}
Recall that for each parameter~$c$ in~$\cP_4(-2)$ the critical value~$c$ of~$f_c$ is in~$P_{c, 4}(- \beta(c))$ (part~$2$ of Lemma~\ref{lem:auxiliary para-puzzle pieces}), so set~$P_{c, 4}(0) = f_c^{-1}(P_{c, 3}(- \beta(c)))$  contains the closure of~$P_{c, 5}(0) = f_c^{-1}(P_{c, 4}(- \beta(c)))$ (\emph{cf}., part~$1$ of Lemma~\ref{lem:puzzle pieces}) and that these sets depend continuously with~$c$ on~$\cP_4(-2)$ (\emph{cf}., Lemma~\ref{lem:hm}).
Since~$\cP_4(-2)$ contains the closure of~$\cP_5(-2)$
(part~$1$ of Lemma~\ref{lem:auxiliary para-puzzle pieces}) we have
$$ A
\=
\inf_{c \in \cP_5(-2)} \modulus (P_{c, 4}(0) \setminus \cl{P_{c, 5}(0)})
>
0. $$
Let~$\Delta_3$ be the constant~$\Delta$ given by Koebe Distortion Theorem for this value of~$A$.

Since~$\hV_c$ is disjoint from the forward orbit of~$0$ (Lemma~\ref{l:off postcritical}), for each connected component~$W$ of~$D_c$ the
map~$f_c^{m_c(W)}$ maps a neighborhood~$\hW$ of~$W$ biholomorphically to~$\hV_c$.
By Koebe Distortion Theorem the distortion of~$f_c^{m_c(W)}$ on~$W$ is bounded by~$\Delta_3$.
Note that~$\hW$ is a puzzle piece intersecting the puzzle piece~$V_c$.
Thus, these puzzle pieces are either equal or one is strictly contained in the other.
Since $\hW$ does not contain~$0$, it follows that~$\hW$ is strictly contained in~$V_c$.
Thus~$\left( f_c^{m_c(W)}|_{\hW} \right)^{-1}$ is an extension of $F_c|_{W}^{-1}$ to~$\hV_c$ taking images in~$V_c$.
\end{proof}
\subsection{Pressure function of the induced map}
\label{ss:Bowen type formula}
Let~$n \ge 4$ be an integer and let~$c$ be a parameter in~$\cK_n$.
In this subsection we state a Bowen type formula relating~$P_c^{\R}$ (resp.~$P_c^{\C}$) to a certain~$2$ variables pressure of~$F_c$ (Proposition~\ref{p:Bowen type formula}) that is shown in~\S\ref{ss:proof of Bowen type formula}.

Denote by~$\fD_c$ the collection of connected components of~$D_c$ and by~$\fD_c^{\R}$ the sub-collection of~$\fD_c$ of those sets intersecting~$\R$.
For each~$W$ in~$\fD_c$ denote by~$\phi_W : \hV_c \to V_c$ the extension of~$F_c|_{W}^{-1}$ given by Lemma~\ref{l:bounded distortion to nice}.
Given an integer~$\ell \ge 1$ denote by~$E_{c, \ell}$ (resp. $E_{c, \ell}^{\R}$) the set of all words of length~$\ell$ in the alphabet~$\fD_c$ (resp. $\fD_c^{\R}$).
So, for each integer~$\ell \ge 1$ and each word~$W_1 \cdots W_\ell$  in~$E_{c, \ell}$ the composition
$$ \phi_{W_1 \cdots W_\ell} = \phi_{W_1} \circ \cdots \circ \phi_{W_\ell} $$
is defined on~$\hV_c$.
Put
$$ m_c(W_1 \cdots W_\ell) = m_c(W_1) + \cdots + m_c(W_\ell). $$

For~$t, p$ in~$\R$ and an integer~$\ell \ge 1$ put
$$ Z_{c, \ell}^{\R}(t, p)
\=
\sum_{\underline{W} \in E_{c, \ell}^{\R}} \exp(-m_c(\underline{W}) p) \left(
\sup \{ |D\phi_{\underline{W}}(z) | \mid z \in V_c \} \right)^t $$
and
$$ Z_{c, \ell}^{\C}(t, p)
\=
\sum_{\underline{W} \in E_{c, \ell}} \exp(-m_c(\underline{W}) p) \left( \sup \{
|D\phi_{\underline{W}}(z) | \mid z \in V_c \} \right)^t. $$
For a fixed~$t$ and~$p$ in~$\R$ the sequence
$$ \left( \frac{1}{\ell} \log Z_{c, \ell}^{\R}(t, p) \right)_{\ell = 1}^{+ \infty}
\left( \text{resp. } \left(\frac{1}{\ell}  \log Z_{c, \ell}^{\C}(t, p) \right)_{\ell = 1}^{+ \infty} \right) $$
converges to the pressure function of~$F_c|_{D_c \cap \R}$ (resp.~$F_c$) for the potential~$- t \log |DF_c| - p m_c$; denote it by~$\sP_c^{\R}(t, p)$ (resp. $\sP_c^{\C}(t, p)$).
On the set where it is finite, the function~$\sP_c^{\R}$ (resp.~$\sP_c^{\C}$) so
defined is strictly decreasing in each of its variables.

\begin{propalph}
\label{p:Bowen type formula}
There is~$n_4 \ge 4$ such that for every integer~$n \ge n_4$ and every parameter~$c$ in~$\cK_n$, we have for each~$t \ge 3$
$$ P_c^{\R}(t)
=
\inf \left\{ p \mid \sP_c^{\R}(t, p) \le 0 \right\}
 \left( \text{resp. } P_c^{\C}(t)
=
\inf \left\{ p \mid \sP_c^{\C}(t, p) \le 0 \right\} \right). $$
\end{propalph}
The proof of this proposition is given in~\S\ref{ss:proof of Bowen type formula}, after we give a lower bound on the pressure function in the next subsection.
\subsection{Critical line}
\label{ss:critical line}
The purpose of this subsection is to prove the following proposition.
\begin{prop}
\label{p:critical line}
For every integer~$n \ge 5$ and every parameter~$c$ in~$\cK_n$ we have
$$ \chiinfR
\=
\inf \left\{ \int \log |Df_c| \ d\mu \mid \mu \in \sM_c^{\R} \right\}
\le \chicrit/2. $$
In particular, for each~$t > 0$ we have
$$ P_c^{\C}(t) \ge P^{\R}_c(t) \ge - t \chicrit/2. $$
\end{prop}

The proof of this proposition is given after the following lemma.

\begin{lemm}
\label{l:critical line}
There is a constant~$C_4 > 0$ such that for each integer~$n \ge 5$ and each parameter~$c$ in~$\cK_n$, the following property holds: For every integer~$k \ge 0$ there is a connected component~$W$ of~$D_c$ contained in~$P_{c, n + 3k + 1}(0)$, that intersects~$\R$ and such that~$m_c(W) = n + 3k + 3$ and
$$ \sup_{z \in W} |DF_c(z)|
\le
C_4 |Df_c^{n + 3k} (c)|^{1 / 2}. $$
\end{lemm}
\begin{proof}
Let~$\Delta_2 > 1$ and~$\Delta_3 > 1$ be the constants given by 
Lemmas~\ref{l:distortion to central 0} 
and~\ref{l:bounded distortion to nice}, respectively.
Since the set~$P_{c, 1}(0)$ depends continuously with~$c$ on~$\cP_0(-2)$ 
(\emph{cf}., Lemma~\ref{lem:hm}) and since this last set contains the closure
of~$\cP_4(-2)$ (part~$1$ of Lemma~\ref{lem:auxiliary para-puzzle pieces}), we have
$$ \Xi_0 \= \sup_{c \in \cP_4(-2)} \diam(P_{c, 1}(0))
<
+ \infty $$
and
$$ \Xi_1 \= \sup_{c \in \cP_4(-2)} \sup_{z \in P_{c, 1}(0)} |Df_c^2 (z)|
<
+ \infty. $$

Fix an integer~$n \ge 5$, a parameter~$c$ in~$\cK_n$, and an integer~$k \ge 0$.
Then the parameter~$c$ is real, so~$\alpha(c)$ and~$\talpha(c)$ are both real (\S\ref{ss:landing to central}) and the set~$\Lambda_c$ is contained in the interval
$$ P_{c, 3}(0) \cap \R = (\gamma(c), \tgamma(c)), $$
see~\S\ref{ss:expanding Cantor set}.
On the other hand, the point~$\alpha_1(c)$ is real (\S\ref{ss:landing to central}) and~$c$ is in~$P_{c, n}(- \beta(c))$ (part~$2$ of Lemma~\ref{lem:auxiliary para-puzzle pieces}), so~$c < \alpha_1(c)$.
Note moreover that~$V_{c, 1}$ is invariant by complex conjugation and that~$V_{c, 1} \cap \R = (\alpha_1(c), \alpha(c))$ (\S\ref{ss:landing to central}).
It follows that~$f_c^{-1}(V_{c, 1})$ is the disjoint union of~$2$ puzzle pieces~$X_c$ and~$\tX_c$, such that
$$ X_c \cap \R = (\alpha(c), \gamma(c))
\text{ and }
\tX_c \cap \R = (\tgamma(c), \talpha(c)). $$
Moreover, each of the puzzle pieces~$X_c$ and~$\tX_c$ is a connected component of~$D_c'$ and~$m_c(X_c) = m_c(\tX_c) = 2$.

Since~$f_c^{n + 3k + 1}$ maps~$P_{c, n + 3k + 2}(0)$ properly onto~$P_{c, 1}(0)$, it follows that~$f_c^{n + 3k + 1}$ maps the end points of the interval~$P_{c, n + 3k + 2}(0) \cap \R$ into~$\partial P_{c, 1}(0) \cap \R = \{ \alpha(c), \talpha(c) \}$.
Since~$f_c^{n + 3k + 1}(0)$ is in~$\Lambda_c \subset Y_c \cup \tY_c$, 
it follows that the interval~$f_c^{n + 3k + 1}(P_{c, n + 3k + 2}(0) \cap \R)$
contains either~$X_c \cap \R$ or~$\tX_c \cap \R$.
This proves that there is a connected component~$W$ of~$D_c$ contained 
in~$P_{n + 3k + 2}(0)$, that intersects~$\R$ and such that~$m_c(W) = n + 3k +
3$. Let~$z_W$ be the unique point in~$W$ such that~$f_c^{n + 3k + 3}(z_W) = 0$.
Then~$f_c^{n + 3k + 1}(z_W)$ belongs to~$P_{c, 1}(0)$, so by definition
of~$\Xi_0$ we have
\begin{equation}
\label{e:image distance estimate}
|f_c^{n + 3k + 1}(z_W) - f_c^{n + 3k}(c)|
\le
\diam(P_{c, 1}(0))
\le
\Xi_0.
\end{equation}
Since~$f_c^n$ maps~$V_{c, n} = P_{c, n + 1}(c)$ biholomorphically to~$P_{c, 1}(0)$ and~$f_c^n(c) \in \Lambda_c$, it follows that~$f_c^{n + 3k}$ maps~$P_{c, n + 3k
+ 1}(c)$ biholomorphically to~$P_{c, 1}(0)$; so the distortion of~$f_c^{n + 3k}$
on~$P_{c, n + 3k + 1}(c)$ is bounded by~$\Delta_2$ (Lemma~\ref{l:distortion to central 0}) and for each point~$y$ in~$P_{c, n + 3k + 1}(c)$ we have
\begin{equation}
\label{e:binding estimate}
\Delta_2^{-1} |Df_c^{n + 3k} (c)|
\le
|Df_c^{n + 3k} (y)|
\le
\Delta_2 |Df_c^{n + 3k} (c)|.
\end{equation}
Together with~\eqref{e:image distance estimate} this implies that,
$$ |f_c(z_W) - c|
\le
\Delta_2 \Xi_0 |Df_c^{n + 3k} (c)|^{-1} $$
and therefore that,
$$ |Df_c(z_W)|
\le
2 \Delta_2^{1/2} \Xi_0^{1/2} |Df_c^{n + 3k} (c)|^{-1/2}. $$
Combined with~\eqref{e:binding estimate} with~$y = f_c(z_W)$, this implies
\begin{equation*}
\label{e:Julia's estimate}
|Df_c^{n + 3k + 1} (z_W)|
\le
2 \Delta_2^{3/2} \Xi_0^{1/2} |Df_c^{n + 3k} (c)|^{1/2}.
\end{equation*}
Putting~$C_4 \= 2 \Delta_3 \Xi_1 \Delta_2^{3/2} \Xi_0^{1/2}$, we get by Lemma~\ref{l:bounded distortion to nice}
\begin{align*}
\sup_{z \in W} |DF_c(z)|
& \le
\Delta_3  |Df_c^{n + 3k + 3} (z_W)|
\\ & \le
\Delta_3 \Xi_1 |Df_c^{n + 3k + 1} (z_W)|
\\ & \le
C_4 |Df_c^{n + 3k} (c)|^{1/2}.
\end{align*}
\end{proof}
\begin{proof}[Proof of Proposition~\ref{p:critical line}]
Let~$C_4$ be given by Lemma~\ref{l:critical line} and for each integer~$k \ge 0$ let~$W_k$ be the element~$W$ of~$\fD_c$ given by the same lemma.
Since~$W_k$ intersects~$\R$ and~$F_c|_{W_k} = f_c^{n + 3k + 3}|_{W_k}$ maps~$W_k$ biholomorphically to~$V_c$, we have~$f_c^{n + 3k + 3}(W_k \cap \R) = V_c \cap \R$.
On the other hand, since~$W_k \subset V_c$, there is a periodic point~$p_k$ of~$f_c$ of period~$n + 3k + 3$ in the closure of~$W_k \cap \R$.
Denoting by~$\mu_k$ the invariant probability measure supported on the orbit of~$p_k$, we have by Lemma~\ref{l:critical line} that for each~$t > 0$
\begin{align*}
\chiinfR
& \le
\int \log |Df_c| \ d\mu_k
\\ & =
\frac{1}{n + 3k + 3} \log |Df_c^{n + 3k + 3} (p_k)|
\\ & \le
\frac{1}{n + 3k + 3} (\log C_4 + \log |Df_c^{n + 3k} (c)|^{1/2}).
\end{align*}
We obtain the desired inequality by letting~$k \to + \infty$.
\end{proof}

\subsection{Proof of Proposition~\ref{p:Bowen type formula}}
\label{ss:proof of Bowen type formula}
For future reference, the following lemma is stated in a stronger form than what is needed for this paper.
\begin{lemm}
\label{l:landing contribution}
There are~$n_5 \ge 4$ and~$C_5 > 1$ such that for every integer~$n \ge n_5$ and every parameter~$c$ in~$\cK_n$ the following property holds: For each~$t \ge 3$, $p \ge - t \chicrit /2 - \tfrac{1}{10} \log 2$, and~$y$ in~$V_c$, we have
$$ L_{t, p}(y)
\=
1 + \sum_{z \in L_c^{-1}(y)} \exp(- m_c(z) p) |DL_c (z)|^{-t}
\le
C_5^t. $$
Moreover, for every integer $\wtm\ge 1$, we have
$$ \sum_{\substack{z \in L_c^{-1}(y), \\ m_c(z) \ge \wtm}} \exp(- m_c(z) p)
|DL_c (z)|^{-t}
\le
C_5^t 2^{-\frac{t}{30}\wtm}. $$
\end{lemm}
\begin{proof}
Put~$\varepsilon_0 \= \tfrac{1}{45}$, let~$C_1$ and~$n_2$ be given by 
Proposition~\ref{p:landing derivatives} with~$\varepsilon = \varepsilon_0$, and let~$n_3$ be given by Lemma~\ref{l:quasi tent estimate} with~$\varepsilon = \varepsilon_0$ and~$m_1 = 4$.
We prove the lemma with~$n_5 \= \max \{n_2, n_3 \}$ and~$C_5 \= C_1 \left( 1 - 2^{- 1/10} \right)^{- 1/3}$.

Let~$n$, $c$, $t$, $p$ and~$y$ be as in the statement of the lemma.
By Lemma~\ref{l:quasi tent estimate} with~$z = f_c^n(c)$, we have
$$ \chicrit
=
\liminf_{m \to + \infty} \frac{1}{m} \log |Df_c^m(c)|
\le
(1 + \varepsilon_0) \log 2. $$
On the other hand, for each integer~$m \ge 1$ the set~$\{ z \in L_c^{-1}(y) \mid
m_c(z) = m \}$ is contained in~$f_c^{-m}(y)$ and therefore it contains at
most~$2^m$ points.
So by Proposition~\ref{p:landing derivatives} and the definition of~$C_5$, 
for every integer $\wtm\ge 1$ we have
$$ \sum_{\substack{z \in L_c^{-1}(y), \\ m_c(z) \ge \wtm} } \exp(- m_c(z) p)
|DL_c (z)|^{-t}
\le
C_1^t \sum_{m = \wtm}^{+ \infty} 2^{m \left( 1 - \tfrac{11}{30} t \right)} 
\le
C_5^t 2^{-\frac{t}{30}\wtm}. $$
\end{proof}

\begin{lemm}
\label{l:Poincare series}
Given an integer~$n \ge 5$ and a parameter~$c$ in~$\cK_n$, the following property holds for every~$t > 0$ and every real number~$p$: If~$\sP_c^{\R}(t, p) > 0$ (resp. $\sP_c^{\C}(t, p) > 0$), then the series
\begin{multline}
  \label{e:Poincare series}
 \sum_{j = 1}^{+ \infty} \exp(- j p) \sum_{y \in f_c|_{I_c}^{-j}(0)} |Df_c^j (y)|^{-t}  
\\ 
\left( \text{resp. $\sum_{j = 1}^{+ \infty} \exp(- j p) \sum_{y \in f_c^{-j}(0)} |Df_c^j (y)|^{-t}$} \right)  
\end{multline}
diverges.
On the other hand, there is~$n_6 \ge 5$ such that if in addition~$n \ge n_6$, then for every~$t \ge 3$ and
$$ p \ge P_c^{\R}(t) - t \tfrac{1}{10} \log 2
\left( \text{resp. $p \ge P_c^{\C}(t) - t \tfrac{1}{10} \log 2$ } \right)$$
satisfying $\sP_c^{\R}(t, p) < 0$ (resp. $\sP_c^{\C}(t, p) < 0$), the series above converges.
\end{lemm}
\begin{proof}
We prove the assertions concerning~$f_c|_{J_c}$; the arguments apply without change to~$f_c|_{I_c}$.
Let~$\Delta_3 > 1$ be the constant given by Lemma~\ref{l:bounded distortion to nice}.

Suppose first~$\sP_c^{\C}(t, p) > 0$.
Since for each integer~$\ell \ge 1$ every point of~$F_c^{-\ell}(0)$ is a preimage of~$0$ by an iterate of~$f_c$, by Lemma~\ref{l:bounded distortion to nice} the series~\eqref{e:Poincare series} is bounded from below by,
\begin{multline*}
\sum_{\ell = 1}^{+ \infty} \sum_{y \in F_c^{-\ell}(0)} \exp( - (m_c(F_c^{\ell - 1}(y)) + \cdots + m_c(y)) p) |DF_c^\ell (y)|^{-t}
\\ \ge
\Delta_3^{-t} \sum_{\ell = 1}^{+ \infty} Z_{c, \ell}^{\C}(t, p)
=
+ \infty.
\end{multline*}

To prove the last part of the lemma, let~$n_5$ and~$C_5 > 1$ be given by Lemma~\ref{l:landing contribution}.
We prove the desired assertion with~$n_6 = n_5$.
Suppose in addition we have~$n \ge n_5$ and let
$$ t \ge 3
\text{ and }
p \ge P_c^{\C}(t) - t \tfrac{1}{10} \log 2 $$
be such that~$\sP_c^{\C}(t, p) < 0$.
By Proposition~\ref{p:critical line} we have~$p \ge - t(\chicrit + \tfrac{1}{5} \log 2)/2$, so~$t$ and~$p$ satisfy the hypotheses of Lemma~\ref{l:landing contribution}.
Given an integer~$m \ge 1$ and a point~$z$ in~$f_c^{-m}(0)$ denote by~$\ell(z)$ the number of those~$j$ in~$\{0, \ldots, m - 1 \}$ such that~$f_c^j(z)$ is in~$V_c$.
In the case where~$z$ is not in~$V_c$, this point is in the domain of~$L_c$ and we have~$\ell(z) = 0$ if and only~$L_c(z) = 0$.
Moreover, if~$z$ is not in~$V_c$ and~$\ell(z) \ge 1$, then~$L_c(z)$ is in the domain of~$F_c^{\ell(z)}$ and~$F_c^{\ell(z)}(L_c(z)) = 0$.
So, if~$z$ is not in~$V_c$ we have in all the cases
$$ |Df_c^m(z)| = |DF_c^{\ell(z)} (L_c(z))| \cdot |DL_c (z)|. $$
Then Lemma~\ref{l:landing contribution} implies that the series~\eqref{e:Poincare series} is bounded from above by
\begin{multline*}
L_{t, p}(0) + \sum_{\ell = 1}^{+ \infty} \sum_{y \in F_c^{- \ell}(0)} L_{t, p}(y) \exp(- (m_c(F_c^{\ell - 1}(y)) + \cdots + m_c(y))p) |DF_c^{\ell} (y)|^{-t}
\\ \le
C_5^t \left( 1 + \sum_{\ell = 0}^{+ \infty} Z_{c, \ell}^{\C}(t, p) \right)
< + \infty.
\end{multline*}
\end{proof}
\begin{proof}[Proof of Proposition~\ref{p:Bowen type formula}]
We prove the assertion for~$f_c|_{J_c}$; the arguments apply without change to~$f_c|_{I_c}$.
Let~$\Delta_3 > 1$ be given by Lemma~\ref{l:bounded distortion to nice} and~$n_6$ by Lemma~\ref{l:Poincare series}.
Let~$n \ge n_6$ be an integer and let~$c$ by a parameter in~$\cK_n$.
We use that fact that for each~$t > 0$ we have
\begin{equation}
\label{e:pressure of original}
P_c^{\C}(t)
=
\limsup_{m \to + \infty} \frac{1}{m} \log \sum_{y \in f_c^{-m}(0)}
|Df_c^m(y)|^{-t},
\end{equation}
see for example~\cite{Urb03c} or~\cite{PrzRivSmi04}.

Fix~$t \ge 3$.
We use the fact that the function~$p \mapsto \sP_c^{\C}(t, p)$ is strictly decreasing where it is finite, see~\S\ref{ss:Bowen type formula}.
In particular, for each~$p$ satisfying~$p < p_0 \= \inf \{ p : \sP_c^{\C}(t, p) \le 0 \}$ we have~$\sP_c^{\C}(t, p) > 0$.
Lemma~\ref{l:Poincare series} implies that for such~$p$ the series~\eqref{e:Poincare series} diverges and by~\eqref{e:pressure of original} we have~$P_c^{\C}(t) \ge p > p_0$.
To prove the reverse inequality, suppose by contradiction~$p_0 <  P_c^{\C}(t)$ and let~$p$ be in the interval~$(p_0, P_c^{\C}(t))$ satisfying~$p \ge P_c^{\C}(t) - t \tfrac{1}{10} \log 2$.
Then~$\sP_c^{\C}(t, p) < 0$ and by Lemma~\ref{l:Poincare series} the series~\eqref{e:Poincare series} converges.
Then~\eqref{e:pressure of original} implies~$P_c^{\C}(t) \le p$ and we obtain a contradiction that completes the proof of the proposition.
\end{proof}

\section{Estimating the geometric pressure function}
\label{s:estimating pressure}
The purpose of this section is to prove the following proposition.
The proof of Proposition~\ref{p:first-order phase transition}, at the end of this section, is based on this proposition, together with Propositions~\ref{p:Bowen type formula} and~\ref{p:critical line}.

Recall that for a real parameter~$c$,
$$ \chicrit = \liminf_{m \to + \infty} \frac{1}{m} \log |Df_c^m(c)|. $$
\begin{propalph}
\label{p:improved MS criterion}
There are~$n_7 \ge 5$ and~$C_6 > 1$ such that for every integer~$n \ge n_7$ and every parameter~$c$ in~$\cK_n$ the following properties hold for each~$t \ge 3$.
\begin{enumerate}
\item[1.]
For~$p$ in~$[- t \chicrit/2, 0)$ satisfying
$$ \sum_{k = 0}^{+ \infty} \exp(- (n + 3k)p)|Df_c^{n + 3k} (c)|^{-t/2}
\ge
C_6^{t}, $$
we have $\sP_c^{\R}(t, p) > 0$ and~$P_c^{\R}(t) \ge p$.
If in addition the sum above is finite, then~$\sP_c^{\C}(t, p)$ is finite and~$P_c^{\R}(t) > p$.
\item[2.]
For~$p \ge - t \chicrit/2$ satisfying
$$ \sum_{k = 0}^{+ \infty} \exp(- (n + 3k)p)|Df_c^{n + 3k} (c)|^{-t/2}
\le
C_6^{-t}, $$
we have $\sP_c^{\C}(t, p) < 0$ and~$P_c^{\C}(t) \le p$.
\item[3.]
For~$p \ge - t \chicrit/2$ satisfying
$$ \sum_{k = 0}^{+ \infty} k \cdot \exp(- (n + 3k)p)|Df_c^{n + 3k} (c)|^{-t/2}
<
+ \infty, $$
we have
$$ \sum_{W \in \fD_c} m_c(W) \cdot \exp(- m_c(W)p) \sup_{z \in W}|DF_c(z)|^{-t}
<
+ \infty. $$
\end{enumerate}
\end{propalph}

The proof of Proposition~\ref{p:improved MS criterion} is given after Lemma~\ref{l:level k contribution}, below, which is used in the proof.
The proof of Proposition~\ref{p:first-order phase transition} is given after the proof of Proposition~\ref{p:improved MS criterion}.

Let~$n \ge 4$ be an integer and~$c$ a parameter in~$\cK_n$.
Since the critical point~$z = 0$ does not belong to~$D_c$ (\emph{cf}., Lemma~\ref{l:off postcritical}), for each integer~$\ell \ge 1$, each connected component of~$D_c$ intersecting~$P_{c, \ell}(0)$ is contained
in~$P_{c, \ell}(0)$.
We define the \emph{level} of a connected component~$W$ of~$D_c$ as the largest integer~$k \ge 0$ such that~$W$ is contained in~$P_{c, n + 3k + 2}(0)$.
Given an integer~$k \ge 0$ denote by~$\fD_{c, k}$ the collection of all connected components of~$D_c$ of level~$k$; we have~$\fD_c = \bigcup_{k = 0}^{+ \infty} \fD_{c, k}$.

For future reference, the following lemma is stated in a stronger form than what is needed for this paper.

\begin{lemm}
\label{l:level k contribution}
There is~$C_7 > 0$ such that for each integer~$n \ge 5$, each parameter~$c$ in~$\cK_n$, each integer~$k \ge 0$, and each pair of real
numbers~$t > 0$ and~$p$, we have
\begin{multline*}
\sum_{W \in \fD_{c, k}} \exp(- m_c(W) p) \sup_{z \in W}|DF_c(z)|^{-t}
\\ \le
2 C_7^t \exp(- (n + 3k + 1) p)|Df_c^{n + 3k} (c)|^{-t/2}
\\ \cdot
\left( 1 + \sum_{w \in L_c^{-1}(0) \text{ in } P_{c, 1}(0)} \exp(- m_c(w) p) |DL_c(w)|^{-t} \right).
\end{multline*}
Moreover, for every integer $\wtm\ge 1$, we have
\begin{multline*}
\sum_{\substack{W \in \fD_{c, k}, \\ m_c(W) \ge \wtm + n + 3k + 1}} \exp(- m_c(W) p) \sup_{z \in W}|DF_c(z)|^{-t}
\\ \le
2 C_7^t \exp(- (n + 3k + 1) p)|Df_c^{n + 3k} (c)|^{-t/2}
\\ \cdot
\left(\sum_{\substack{w \in L_c^{-1}(0) \text{ in } P_{c, 1}(0), \\ m_c(w) \ge \wtm}} \exp(- m_c(w) p) |DL_c(w)|^{-t} \right).
\end{multline*}
\end{lemm}
\begin{proof}
Let~$\Delta_2$, $C_3$, and~$\Delta_3$ be the constants given by 
Lemmas~\ref{l:distortion to central 0}, \ref{l:first return to central derivative} and~\ref{l:bounded distortion to nice}, respectively.

Fix an integer~$n \ge 5$, a parameter~$c$ in~$\cK_n$, and an integer~$k \ge 0$.
Note that there are precisely~$2$ elements~$W'$ of~$\fD_{c, k}$ such
that~$m_c(W') = n + 3k + 1$; denote them by~$W_0$ and~$W_0'$.
Indeed, these sets are the connected components of the  
preimage under $f_c$ of the set 
$(f_c^{n+3k}|_{P_{c,n+3k + 1}(c)})^{-1}(V_c)$.
For a connected component~$W$ of~$D_c$ of level~$k$ denote by~$z_W$ the unique point in~$W$ such that~$F_c(z_W) = 0$.
If~$W$ is different from~$W_0$ and~$W_0'$, then~$z' \= f_c^{n + 3k + 1}(z_W)$ is
different from~$0$ and it is in the domain of definition~$D_c'$ of the first landing map~$L_c$ to~$V_c$.
So, denoting by~$W'$ the connected component of~$D_c'$ containing~$z'$, there is a unique point~$w_W$ in~$W'$ such that~$L_c(w_W) = 0$ and we have
$$ m_c(W) = n + 3k + 1 + m_c(w_W). $$

Since~$f_c^n$ maps~$V_{c, n} = P_{c, n + 1}(c)$ biholomorphically to~$P_{c, 1}(0)$ and~$f_c^n(c)$ is in~$\Lambda_c$, it follows that~$f_c^{n + 3k}$ maps~$P_{c, n + 3k + 1}(c)$ biholomorphically to~$P_{c, 1}(0)$; so the distortion of~$f_c^{n + 3k}$ on~$P_{c, n + 3k + 1}(c)$ is bounded by~$\Delta_2$ (Lemma~\ref{l:distortion to central 0}) and for each point~$y$ in~$P_{c, n + 3k + 1}(c)$ we have
\begin{equation*}
|Df_c^{n + 3k}(y)|
\ge
\Delta_2^{-1} |Df_c^{n + 3k} (c)|.
\end{equation*}
On the other hand, by part~$1$ of Lemma~\ref{l:first return to central derivative} with~$x = z_W$ and~$q = n + 3k + 1$, we have
\begin{equation*}
|Df_c^{n + 3k + 1} (z_W)|
\ge
C_3^{-1} |Df_c^{n + 3k} (f_c(z_W))|^{1/2}
\ge
C_3^{-1} \Delta_2^{-1/2} |Df_c^{n + 3k} (c)|^{1/2}.
\end{equation*}
Together with the inequality,
$$ |Df_c^{m_c(w_W)} (f_c^{n + 3k + 1}(z_W))|
\ge
\Delta_2^{-1} |DL_c(w_W)| $$
given by Lemmas~\ref{l:distortion to central 0} and~\ref{l:univalent pull-back property}, this implies that if we put $C_0 \= C_3^{-1} \Delta_2^{- 3/2}$, then
\begin{equation*}
|DF_c(z_W)|
=
|Df_c^{m_c(W)} (z_W)|
\ge
C_0 |Df_c^{n + 3k} (c)|^{1/2}|DL_c(w_W)|.
\end{equation*}
Since the distortion of~$F_c|_W$ is bounded by~$\Delta_3$ (Lemma~\ref{l:bounded distortion to nice}), for each~$p > 0$ and~$t > 0$, we have
\begin{multline}
\label{e:single component upper estimate}
\exp(- m_c(W) p) \sup_{z \in W} |DF_c(z)|^{-t}
\\ \le
\Delta_3^t C_0^{- t}
\exp(- (n + 3k + 1) p)|Df_c^{n + 3k} (c)|^{-t/2}
\cdot
\exp(- m_c(w_W)p ) |DL_c(w_W)|^{-t}.
\end{multline}
To prove the desired inequality, observe that for each point~$w$ 
of~$L_c^{-1}(0)$ in~$P_{c, 1}(0)$ there are precisely~$2$ connected
components~$W$ of~$D_c$ in $\fD_{c,k}$ such that~$w_W = w$; in fact for each
such~$W$ the set~$f_c(W)$ is uniquely determined as the preimage by the
univalent map~$f_c^{n + 3k}|_{P_{c, n + 3k + 1}(c)}$ of the connected component
of~$D_c'$ containing~$w_W$.
Thus, the desired inequalities follow from~\eqref{e:single component upper estimate} with~$C_7 = \Delta_3 C_0^{-1}$.
\end{proof}

\begin{lemm}
\label{l:estimating Z_1 by the postcritical series}
There are~$n_8 \ge 5$ and~$C_8 > 1$ such that for every integer~$n \ge n_8$
and 
every parameter~$c$ in~$\cK_n$, the following properties hold for each~$t \ge 3$
and each integer $k \ge 0$:
\begin{enumerate}
\item[1.]
For each~$p < 0$, we have
\begin{multline*} 
\sum_{W \in \fD_{c,k}\cap \fD_c^\R} \exp(-m_c(W)p) \inf_{z\in W} |DF_c(z)|^{-t}
>  \\
C_8^{-t} \exp(- (n + 3k)p)|Df_c^{n + 3k} (c)|^{-t/2}.
\end{multline*}
\item[2.]
For each~$p \ge - t \chicrit/2 - t \frac{1}{10} \log 2$, we have
\begin{multline*} 
\sum_{W \in
\fD_{c,k}} \exp(-m_c(W)p) \sup_{z\in W} |DF_c(z)|^{-t}
<  \\
C_8^{t} \exp(- (n + 3k)p)|Df_c^{n + 3k} (c)|^{-t/2}.
\end{multline*}
\end{enumerate}
\end{lemm}

\proof
Let~$C_2$ and $n_3$ be given by Lemma~\ref{l:quasi tent estimate} with~$m_1 = 4$ and $\varepsilon = \frac{1}{10}$, let~$n_5$ and~$C_5 > 0$ be given by Lemma~\ref{l:landing contribution}, and let~$C_4$ and~$C_7$ be given by Lemmas~\ref{l:critical line} and~\ref{l:level k contribution}, respectively.
We prove the lemma with~$n_8 \= \max \{n_3, n_4, n_5 \}$.
To do this, fix an integer~$n \ge n_8$, a parameter~$c$ in~$\cK_n$, $t \ge 3$, and an integer $k \ge 0$.

To prove part~$1$, let~$W_k$ be the component~$W$ of~$D_c$ given Lemma~\ref{l:critical line}.
Then for each $p < 0$, we have
\begin{multline*}
\sum_{W \in \fD_{c,k}\cap \fD_c^\R} \exp(-m_c(W)p) \inf_{z\in W} |DF_c(z)|^{-t}
\\ 
\begin{aligned}
& \ge  
\exp( - m_c(W_k) p) \inf_{z \in W_k} |DF_c(z)|^{-t}
\\ & \ge
C_4^{-t} \exp( - (n + 3k + 3) p) |Df_c^{n +
3k}(c)|^{-t/2}
\\ & >
C_4^{-t} \exp( - (n + 3k) p) |Df_c^{n +
3k}(c)|^{-t/2}.
\end{aligned}
\end{multline*}
This proves part~$1$ of the lemma with~$C_8 = C_4$.

To prove part~$2$, let~$p \ge - t\chicrit/2 - t \frac{1}{10} \log 2$ be given.
By Lemma~\ref{l:quasi tent estimate} with~$z = f_c^n(c)$, we have
\begin{equation*}
\chicrit
=
\liminf_{m \to + \infty} \frac{1}{m} \log |Df_c^m(c)|
\le
\tfrac{11}{10} \log 2.
\end{equation*}
Thus~$p \ge - t \tfrac{13}{20} \log 2$ and therefore $2 \exp(- p) < 2^t$.
Combined with Lemmas~\ref{l:landing contribution} and~\ref{l:level k contribution}, we obtain part~$2$ of the lemma with~$C_8 = 2 C_7 C_5$.
\endproof

\begin{proof}[Proof of Proposition~\ref{p:improved MS criterion}]
Let $n_4$ be given by Proposition~\ref{p:Bowen type formula} and let~$n_8$ and~$C_8$ be given by Lemma~\ref{l:estimating Z_1 by the postcritical series}.
To prove the proposition, fix an integer~$n \ge \max \{ n_4, n_8 \}$, a parameter~$c$ in~$\cK_n$, and~$t \ge 3$.

To prove part~$1$, let~$p$ be in~$[ - t \chicrit / 2, 0)$.
By part~$1$ of Lemma~\ref{l:estimating Z_1 by the postcritical series}, if the sum
\begin{equation}
  \label{e:postcritical series}
\sum_{k = 0}^{+ \infty} \exp( - (n + 3k) p) |Df_c^{n + 3k}(c)|^{-t/2}
\end{equation}
is greater than or equal to~$C_8^t$, then~$\sP_c^{\R}(t, p) > 0$ and by Proposition~\ref{p:Bowen type formula} we have~$P_c^{\R}(t) \ge p$.
This proves the first part of part~$1$ with~$C_6 = C_8$.
To complete the proof of part~$1$, suppose~\eqref{e:postcritical series} is finite and greater than or equal to~$(2C_8)^t$.
Then there is~$p' > p$ such that~\eqref{e:postcritical series} with~$p$ replaced by~$p'$ is greater than or equal to~$C_8^t$.
As shown above, this implies~$P_c^{\R}(t) \ge p' > p$.
On the other hand, by part~$2$ of Lemma~\ref{l:estimating Z_1 by the postcritical series} the sum
$$ \sum_{W \in \fD_c} \exp(- m_c(W)p) \sup_{z \in W}|DF_c(z)|^{-t} $$
is finite, so~$\sP_c^{\R}(t, p)$ is also finite.
This completes the proof of part~$1$ with~$C_6 = 2C_8$.

To prove part~$2$, let~$p \ge - t\chicrit/2$ be given.
By part~$2$ of Lemma~\ref{l:estimating Z_1 by the postcritical series}, if~\eqref{e:postcritical series} is less than or equal to~$C_8^{-t}$, then~$\sP_c^{\C}(t, p) < 0$ and by Proposition~\ref{p:Bowen type formula} we have~$P_c^{\C}(t) \le p$.
This proves part~$2$ of the proposition with~$C_6 = C_8$.

To prove part~$3$, let~$p \ge - t \chicrit/2$ be given and put~$p' \= p - t \tfrac{1}{10} \log 2$.
By part~$2$ of Lemma~\ref{l:estimating Z_1 by the postcritical series} with~$k = 0$, the sum
$$ \sum_{W \in \fD_{c, 0}} \exp(- m_c(W) p) \sup_{z \in W}|DF_c(z)|^{-t} $$
is finite.
Let~$A > 0$ be a constant such that for every pair of integers~$k \ge 1$ and~$m \ge 3k + 1$, we have
$$ m
\le
A k 2^{t (m - 3k)/10}. $$
Applying part~$2$ of Lemma~\ref{l:estimating Z_1 by the postcritical series} with~$p$ replaced by~$p'$, we obtain that for each integer~$k \ge 1$ we have
\begin{multline*}
\sum_{W \in \fD_{c, k}} m_c(W) \cdot \exp(- m_c(W) p) \sup_{z \in W}|DF_c(z)|^{-t}
\\
\begin{aligned}
& \le
\sum_{W \in \fD_{c, k}} A k 2^{t (m_c(W) - 3k) / 10} \exp(- m_c(W) p) \sup_{z \in W}|DF_c(z)|^{-t}
\\ & =
A k 2^{- 3k t/10} \sum_{W \in \fD_{c, k}} \exp \left( - m_c(W) p' \right) \sup_{z \in W}|DF_c(z)|^{-t}
\\ & \le
\left(A C_8^t 2^{n t/10} \right) k \cdot \exp(- (n + 3k)p) |Df_c^{n + 3k}(c)|^{-t/2}.
\end{aligned}
\end{multline*}
Summing over~$k \ge 0$ we obtain the desired assertion.
\end{proof}

\begin{proof}[Proof of Proposition~\ref{p:first-order phase transition}]
We give the proof for~$f_c|_{J_c}$; the proof for~$f_c|_{I_c}$ is analogous.
Let~$n_4$ be given by Proposition~\ref{p:Bowen type formula} and let~$n_7$ and $C_6$ be given by Proposition~\ref{p:improved MS criterion}.
Put
$$ n_0 \= \max \{ 5, n_4, n_7 \}
\text{ and }
C_0 \= C_6 $$
and let~$n \ge n_0$ be an integer and~$c$ a parameter in~$\cK_n$ for which the hypotheses of the proposition are satisfied.

The first hypothesis of the proposition together with part~$2$ of Proposition~\ref{p:improved MS criterion} with~$p = - t \chicrit/2$ imply that for every sufficiently large~$t > 0$, we have
$$ P_c^{\C}(t) \le - t \chicrit/2. $$
From Proposition~\ref{p:critical line} we deduce that for such~$t$ we have equality.
The second hypothesis of the proposition together with part~$1$ of
Proposition~\ref{p:improved MS criterion} with~$t = t_0$ and~$p = - t_0 \chicrit/2$, imply that~$f_c$ has a phase transition at some~$t_* > t_0$ satisfying
$$ P_c^{\C}(t_*) = - t_* \chicrit/2 < 0. $$
This proves the first part of the proposition.

To prove the second part of the proposition, we first prove that there exists an equilibrium state of~$f_c$ for the potential~$- t_* \log |Df_c|$.
Our additional hypothesis together with part~$3$ of Proposition~\ref{p:improved MS criterion} with~$t = t_*$ and~$p = - t_* \chicrit /2$, imply that
\begin{equation}
\label{e:time integrability}
\sum_{W \in \fD_c} m_c(W) \exp(m_c(W) t_* \chicrit/2) \sup_{z \in W} |DF_c(z)|^{- t_*}
<
+ \infty.
\end{equation}
The rest of the argument is now standard; we refer to~\cite[\S$4$]{PrzRiv11} for precisions, see also Remark~\ref{r:chiinf} below.
Since~$\sP_c^{\C}(t_*, - t_* \chicrit/2) = 0$, there is a $(t_*, - t_*
\chicrit/2)$\nobreakdash-conformal measure~$\mu$ for~$f_c$ that assigns positive measure to the maximal invariant set of~$F_c$, see~\cite[Theorem~A in~\S$4$ and Proposition~$4.3$]{PrzRiv11}.
Standard considerations imply that there is an invariant probability measure~$\rho$ for~$F_c$ that is absolutely continuous with respect to~$\mu$.
Thus~\eqref{e:time integrability} together with the bounded distortion property of~$F_c$ (Lemma~\ref{l:bounded distortion to nice}) imply that the sum $\sum_{W \in \fD_c} m_c(W) \rho(W)$ is finite.
Therefore the measure
$$ \hrho \= \sum_{W \in \fD_c} \sum_{j = 0}^{m_c(W) - 1} (f_c^j)_* (\rho|_W) $$ 
is finite.
This measure is invariant by~$f_c$ and it is absolutely continuous with respect to~$\mu$.
To prove that the probability measure proportional to~$\hrho$ is an equilibrium state of~$f_c$ for the potential~$- t \log |Df_c|$, first remark that~$\rho$ is an equilibrium state of~$F_c$ for the potential~$- t_* \log |DF_c| + (t_* \chicrit /2) m_c$ and that the measure-theoretic entropy of this measure is strictly positive, see for example~\cite{MauUrb03}.
Then the generalized Abramov formula~\cite[Proposition~$5.1$]{Zwe05} implies that the measure-theoretic entropy of~$\hrho$ is strictly positive and that the probability measure proportional to~$\hrho$ is an equilibrium state of~$f_c$ for the potential~$- t_* \log |Df_c|$.
That this measure is exact, and hence ergodic and mixing, is shown for example in~\cite{You99}.
Finally, the uniqueness of the equilibrium state is given by Ruelle's inequality and by~\cite[Theorem~$6$]{Dob1304} in the real setting and~\cite[Theorem~$8$]{Dob12} in the complex setting.

The non-differentiability of~$P_c^{\C}$ at~$t = t_*$ follows from the existence of an equilibrium state of~$f_c$ for the potential~$- t_* \log |Df_c|$, see~\cite[Corollary~$1.3$]{InoRiv12}.
\end{proof}
\begin{rema}
\label{r:chiinf}
For completeness we show that for a parameter~$c$ as in the proof of Proposition~\ref{p:first-order phase transition} we have~$\chiinfR = \chicrit/2$ and
$$ \chiinfC \= \inf \left\{ \int \log |Df_c| \ d\mu \mid \mu \in \sM_c^{\C} \right\}
=
\chicrit/2, $$
although this is not needed in the proof.
For each invariant probability measure~$\mu$ of~$f_c$ supported on~$J_c$ and every sufficiently large~$t > 0$ we have
$$ \int \log |Df_c| \ d\mu \ge (- P_c^{\C}(t) + h_\mu(f_c))/t
=
\chicrit/2 + h_\mu(f_c)/t
\ge
\chicrit/2. $$
This proves~$\chiinfC \ge \chicrit/2$.
Together with Proposition~\ref{p:critical line} and with the inequality~$\chiinfR \ge \chiinfC$, this implies~$\chiinfR = \chiinfC = \chicrit/2$.
\end{rema}

\appendix
\section{Multipliers of periodic orbits of period~$3$} 
\label{s:appendix}
This appendix is devoted to prove the following lemma, used in~\S\ref{ss:proof of Main Theorem}.
The functions~$p$ and~$\wtp$ appearing in the following lemma are defined in~\S\ref{ss:proof of Main Theorem}.
\begin{lemm}
\label{l:period 3 transversality}
We have,
$$ \left. \frac{\partial}{\partial c} |Df_c^3(p(c))| \right|_{c = - 2}
> 
\left. \frac{\partial}{\partial c} |Df_c^3(\wtp(c))| \right|_{c = - 2}. $$
\end{lemm}
\begin{proof}
Notice that for~$c = -2$
$$ \cO \= \{ 2 \cos (2\pi/7), 2 \cos (4 \pi/ 7), 2 \cos(6 \pi/7) \} $$
and
$$ \widetilde{\cO} \= \{ 2 \cos (2\pi/9), 2 \cos (4 \pi/ 9), 2 \cos(8 \pi/9) \} $$
are the only periodic orbits of minimal period~$3$ of~$f_{-2}$.
Since,
$$ P_{-2, 1}(0) \cap \R
=
[\alpha(-2), \talpha(-2)]
=
[-1, 1], $$
it follows that~$x = 2\cos(4 \pi / 7)$ and~$x = 2 \cos( 4 \pi / 9)$ are the only periodic points of period~$3$ of~$f_{-2}$ in~$P_{-2, 1}(0)$.
On the other hand, the inequalities
\[
2 \cos(4\pi/7) ~<~ 0 ~< ~2 \cos(4\pi/9)
\]
imply that~$p(-2) = 2 \cos(4 \pi / 7)$ and~$\wtp(-2) = 2 \cos(4 \pi / 9)$ and that
$$ Df_{-2}^3 (p(-2)) > 0 > Df_{-2}^3 (\wtp(-2)). $$
Since both functions~$p$ and~$\wtp$ are real, the desired assertion is equivalent to,
\begin{equation}
\label{e:multiplier divergence}
\left. \frac{\partial}{\partial c} Df_c^3(p(c)) \right|_{c = - 2}
>
- \left. \frac{\partial}{\partial c} Df_c^3(\wtp(c)) \right|_{c = - 2}.
\end{equation}

Let~$\pi_0$ be either one of the functions~$p$, $f_c \circ p$, $f_c^2 \circ p$, $\wtp$, $f_c \circ \wtp$, or~$f_c^2 \circ \wtp$ and put
$$ \pi_1(c) = f_c \circ \pi_0(c)
\text{ and }
\pi_2(c) = f_c \circ \pi_1(c). $$
Then~$f_c(\pi_2(c)) = \pi_0(c)$ and a direct computation shows that
$$ D\pi_0
=
- \frac{1 + 2 \pi_2 + 4 \pi_1 \pi_2}{8 \pi_0 \pi_1 \pi_2 - 1}. $$
Therefore, for each~$c$ in~$\cP_3(-2)$ we have,
\begin{align*}
& ((D\pi_0) \pi_1 \pi_2)(c)
\\ & =
- \frac{\pi_1(c) \pi_2(c) + 2 \pi_0(c) \pi_1(c) + 4 \pi_0(c) \pi_2(c) - 2 c \pi_1 (c)- 4 c \pi_0(c) - 4 c \pi_2(c) + 4 c^2}{8 \pi_0(c) \pi_1(c) \pi_2(c) - 1}.
\end{align*}
Using the formula above and the formula above with~$\pi_0$ replaced by~$\pi_1$ and then by~$\pi_2$, we obtain
\begin{align*}
& \frac{\partial}{\partial c} Df_c^3(\pi_0(c))
\\ & =
8 D (\pi_0 \pi_1 \pi_2)(c)
\\ & =
- 8 \frac{7 (\pi_0(c) \pi_1(c) + \pi_1(c) \pi_2(c) + \pi_2(c) \pi_0(c)) - 10 c (\pi_0(c) + \pi_1(c) + \pi_2(c)) + 12 c^2}{8 \pi_0(c) \pi_1(c) \pi_2(c) - 1}.
\end{align*}
Thus, if for each~$j$ in~$\{ 1, 2, 3 \}$ we denote by~$\sigma_j$ (resp. $\tsigma_j$) the elementary symmetric function of degree~$j$ in the elements of~$\cO$ (resp.~$\widetilde{\cO}$), then by the above equation with~$\pi_0 = p$ (resp. $\pi_0 = \wtp$) and~$c = -2$ we obtain,
$$
\left. \frac{\partial}{\partial c} Df_c^3(p(c)) \right|_{c = - 2}
=
- 8 \frac{7 \sigma_2 + 20 \sigma_1 + 48}{8 \sigma_3 - 1}
$$
$$ \left( \text{resp.}
\left. \frac{\partial}{\partial c} Df_c^3(\wtp(c)) \right|_{c = - 2}
=
- 8 \frac{7 \tsigma_2 + 20 \tsigma_1 + 48}{8 \tsigma_3 - 1}
\right).$$

To calculate these numbers, for a given an integer~$n \ge 2$ let~$T_n$ be the $n$-th Chebyshev polynomial, so that for every real number~$\theta$ we have
$$ T_n(\cos(\theta)) = \cos(n \theta). $$
Notice that the zeros of the polynomial~$T_4(x/2) - T_3(x/2)$ different from~$x = 2$ are precisely the elements of~$\cO$.
We thus have the identity
$$ \frac{2 T_4(x/2) - 2 T_3(x/2)}{x - 2}
=
x^3 + x^2 - 2 x - 1
=
x^3 - \sigma_1 x^2 + \sigma_2 x - \sigma_3. $$
So~$\sigma_1 = -1$, $\sigma_2 = -2$, $\sigma_3 = 1$ and by the above
$$ \left. \frac{\partial}{\partial c} Df_c^3(p(c)) \right|_{c = - 2}
=
- 16.
$$

On the other hand, the zeros of the polynomial~$T_5(x/2) - T_4(x/2)$ different from~$x = 2$ and~$x = -1$ are precisely the elements of~$\widetilde{\cO}$.
Therefore we have the identity
$$ \frac{2 T_5(x/2) - 2 T_4(x/2)}{(x - 2)(x + 1)}
=
x^3 - 3 x + 1
=
x^3 - \tsigma_1 x^2 + \tsigma_2 x - \tsigma_3. $$
So~$\tsigma_1 = 0$, $\tsigma_2 = - 3$, $\tsigma_3 = - 1$ and
$$ \left. \frac{\partial}{\partial c} Df_c^3(\wtp(c)) \right|_{c = - 2}
=
24.
$$
This proves~\eqref{e:multiplier divergence} and completes the proof of the lemma.
\end{proof}

\bibliographystyle{alpha}

\begin{thebibliography}{dFdM08}

\bibitem[AM05]{AviMor05}
Artur Avila and Carlos~Gustavo Moreira.
\newblock Statistical properties of unimodal maps: the quadratic family.
\newblock {\em Ann. of Math. (2)}, 161(2):831--881, 2005.

\bibitem[BC85]{BenCar85}
Michael Benedicks and Lennart Carleson.
\newblock On iterations of {$1-ax\sp 2$} on {$(-1,1)$}.
\newblock {\em Ann. of Math. (2)}, 122(1):1--25, 1985.

\bibitem[BMS03]{BinMakSmi03}
I.~Binder, N.~Makarov, and S.~Smirnov.
\newblock Harmonic measure and polynomial {J}ulia sets.
\newblock {\em Duke Math. J.}, 117(2):343--365, 2003.

\bibitem[Bow75]{Bow75}
Rufus Bowen.
\newblock {\em Equilibrium states and the ergodic theory of {A}nosov
  diffeomorphisms}.
\newblock Lecture Notes in Mathematics, Vol. 470. Springer-Verlag, Berlin,
  1975.

\bibitem[BT09]{BruTod09}
Henk Bruin and Mike Todd.
\newblock Equilibrium states for interval maps: the potential {$-t\log\vert
  Df\vert $}.
\newblock {\em Ann. Sci. \'Ec. Norm. Sup\'er. (4)}, 42(4):559--600, 2009.

\bibitem[CG93]{CarGam93}
Lennart Carleson and Theodore~W. Gamelin.
\newblock {\em Complex dynamics}.
\newblock Universitext: Tracts in Mathematics. Springer-Verlag, New York, 1993.

\bibitem[CRL13]{CorRivb}
Daniel Coronel and Juan Rivera-Letelier.
\newblock High-order phase transitions in the quadratic family.
\newblock 2013.
\newblock arXiv:1305.4971v1.

\bibitem[dFdM08]{dFadMe08}
Edson de~Faria and Welington de~Melo.
\newblock {\em Mathematical tools for one-dimensional dynamics}, volume 115 of
  {\em Cambridge Studies in Advanced Mathematics}.
\newblock Cambridge University Press, Cambridge, 2008.

\bibitem[DH84]{DouHub84}
A.~Douady and J.~H. Hubbard.
\newblock {\em \'{E}tude dynamique des polyn\^omes complexes. {P}artie {I}},
  volume~84 of {\em Publications Math\'ematiques d'Orsay [Mathematical
  Publications of Orsay]}.
\newblock Universit\'e de Paris-Sud, D\'epartement de Math\'ematiques, Orsay,
  1984.

\bibitem[dMvS93]{dMevSt93}
Welington de~Melo and Sebastian van Strien.
\newblock {\em One-dimensional dynamics}, volume~25 of {\em Ergebnisse der
  Mathematik und ihrer Grenzgebiete (3) [Results in Mathematics and Related
  Areas (3)]}.
\newblock Springer-Verlag, Berlin, 1993.

\bibitem[Dob09]{Dob09}
Neil Dobbs.
\newblock Renormalisation-induced phase transitions for unimodal maps.
\newblock {\em Comm. Math. Phys.}, 286(1):377--387, 2009.

\bibitem[Dob12]{Dob12}
Neil Dobbs.
\newblock Measures with positive {L}yapunov exponent and conformal measures in
  rational dynamics.
\newblock {\em Trans. Amer. Math. Soc.}, 364(6):2803--2824, 2012.

\bibitem[Dob13]{Dob1304}
Neil Dobbs.
\newblock Pesin theory and equilibrium measures on the interval.
\newblock arXiv:1304.3305v1, 2013.

\bibitem[Gou04]{gouezelthesis}
Sebastien Gouez{\"e}l.
\newblock {\em Vitesse de d{\'e}corr{\'e}lation et th{\'e}or{\`e}mes limites
  pour les applications non uniform{\'e}ment dilatantes}.
\newblock PhD thesis, 2004.

\bibitem[GPR10]{GelPrzRam10}
Katrin Gelfert, Feliks Przytycki, and Micha{\l} Rams.
\newblock On the {L}yapunov spectrum for rational maps.
\newblock {\em Math. Ann.}, 348(4):965--1004, 2010.

\bibitem[GS12]{GaoShe1111}
Bing Gao and Weixiao Shen.
\newblock Summability implies {C}ollet-{E}ckmann almost surely.
\newblock 2012.
\newblock arXiv:1111.3720v3, to appear in \emph{Ergodic Theory Dynam. Systems}.

\bibitem[Hub93]{Hub93}
J.~H. Hubbard.
\newblock Local connectivity of {J}ulia sets and bifurcation loci: three
  theorems of {J}.-{C}. {Y}occoz.
\newblock In {\em Topological methods in modern mathematics ({S}tony {B}rook,
  {NY}, 1991)}, pages 467--511. Publish or Perish, Houston, TX, 1993.

\bibitem[IRRL12]{InoRiv12}
Irene Inoquio-Renteria and Juan Rivera-Letelier.
\newblock A characterization of hyperbolic potentials of rational maps.
\newblock {\em Bull. Braz. Math. Soc. (N.S.)}, 43(1):99--127, 2012.

\bibitem[IT11]{IomTod11}
Godofredo Iommi and Mike Todd.
\newblock Dimension theory for multimodal maps.
\newblock {\em Ann. Henri Poincar\'e}, 12(3):591--620, 2011.

\bibitem[KN92]{KelNow92}
Gerhard Keller and Tomasz Nowicki.
\newblock Spectral theory, zeta functions and the distribution of periodic
  points for {C}ollet-{E}ckmann maps.
\newblock {\em Comm. Math. Phys.}, 149(1):31--69, 1992.

\bibitem[LV73]{LehVir73}
O.~Lehto and K.~I. Virtanen.
\newblock {\em Quasiconformal mappings in the plane}.
\newblock Springer-Verlag, New York, second edition, 1973.
\newblock Translated from the German by K. W. Lucas, Die Grundlehren der
  mathematischen Wissenschaften, Band 126.

\bibitem[Ma{\~n}93]{Man93}
Ricardo Ma{\~n}{\'e}.
\newblock On a theorem of {F}atou.
\newblock {\em Bol. Soc. Brasil. Mat. (N.S.)}, 24(1):1--11, 1993.

\bibitem[McM94]{McM94}
Curtis~T. McMullen.
\newblock {\em Complex dynamics and renormalization}, volume 135 of {\em Annals
  of Mathematics Studies}.
\newblock Princeton University Press, Princeton, NJ, 1994.

\bibitem[Mil00]{Mil00c}
John Milnor.
\newblock Periodic orbits, externals rays and the {M}andelbrot set: an
  expository account.
\newblock {\em Ast\'erisque}, (261):xiii, 277--333, 2000.
\newblock G{\'e}om{\'e}trie complexe et syst{\`e}mes dynamiques (Orsay, 1995).

\bibitem[Mil06]{Mil06}
John Milnor.
\newblock {\em Dynamics in one complex variable}, volume 160 of {\em Annals of
  Mathematics Studies}.
\newblock Princeton University Press, Princeton, NJ, third edition, 2006.

\bibitem[Mis81]{Mis81}
Micha{\l} Misiurewicz.
\newblock Absolutely continuous measures for certain maps of an interval.
\newblock {\em Inst. Hautes \'Etudes Sci. Publ. Math.}, (53):17--51, 1981.

\bibitem[MS00]{MakSmi00}
N.~Makarov and S.~Smirnov.
\newblock On ``thermodynamics'' of rational maps. {I}. {N}egative spectrum.
\newblock {\em Comm. Math. Phys.}, 211(3):705--743, 2000.

\bibitem[MS03]{MakSmi03}
N.~Makarov and S.~Smirnov.
\newblock On thermodynamics of rational maps. {II}. {N}on-recurrent maps.
\newblock {\em J. London Math. Soc. (2)}, 67(2):417--432, 2003.

\bibitem[MT88]{MilThu88}
John Milnor and William Thurston.
\newblock On iterated maps of the interval.
\newblock In {\em Dynamical systems ({C}ollege {P}ark, {MD}, 1986--87)}, volume
  1342 of {\em Lecture Notes in Math.}, pages 465--563. Springer, Berlin, 1988.

\bibitem[MU03]{MauUrb03}
R.~Daniel Mauldin and Mariusz Urba{\'n}ski.
\newblock {\em Graph directed Markov systems}, volume 148 of {\em Cambridge
  Tracts in Mathematics}.
\newblock Cambridge University Press, Cambridge, 2003.
\newblock Geometry and dynamics of limit sets.

\bibitem[NS98]{NowSan98}
Tomasz Nowicki and Duncan Sands.
\newblock Non-uniform hyperbolicity and universal bounds for {$S$}-unimodal
  maps.
\newblock {\em Invent. Math.}, 132(3):633--680, 1998.

\bibitem[Pes97]{Pes97}
Yakov~B. Pesin.
\newblock {\em Dimension theory in dynamical systems}.
\newblock Chicago Lectures in Mathematics. University of Chicago Press,
  Chicago, IL, 1997.
\newblock Contemporary views and applications.

\bibitem[PRL11]{PrzRiv11}
Feliks Przytycki and Juan Rivera-Letelier.
\newblock Nice inducing schemes and the thermodynamics of rational maps.
\newblock {\em Comm. Math. Phys.}, 301(3):661--707, 2011.

\bibitem[PRL13]{PrzRivinterval}
Feliks Przytycki and Juan Rivera-Letelier.
\newblock Geometric pressure for multimodal maps of the interval.
\newblock 2013.
\newblock Preliminary version available at
  \texttt{http://www.impan.pl/$\sim$feliksp/interval24a.pdf}.

\bibitem[PRLS03]{PrzRivSmi03}
Feliks Przytycki, Juan Rivera-Letelier, and Stanislav Smirnov.
\newblock Equivalence and topological invariance of conditions for non-uniform
  hyperbolicity in the iteration of rational maps.
\newblock {\em Invent. Math.}, 151(1):29--63, 2003.

\bibitem[PRLS04]{PrzRivSmi04}
Feliks Przytycki, Juan Rivera-Letelier, and Stanislav Smirnov.
\newblock Equality of pressures for rational functions.
\newblock {\em Ergodic Theory Dynam. Systems}, 24(3):891--914, 2004.

\bibitem[PS08]{PesSen08}
Yakov Pesin and Samuel Senti.
\newblock Equilibrium measures for maps with inducing schemes.
\newblock {\em J. Mod. Dyn.}, 2(3):397--430, 2008.

\bibitem[PU10]{PrzUrb10}
Feliks Przytycki and Mariusz Urba{\'n}ski.
\newblock {\em Conformal fractals: ergodic theory methods}, volume 371 of {\em
  London Mathematical Society Lecture Note Series}.
\newblock Cambridge University Press, Cambridge, 2010.

\bibitem[RL12]{Riv1204}
Juan Rivera-Letelier.
\newblock Asymptotic expansion of smooth interval maps.
\newblock 2012.
\newblock arXiv:1204.3071v2.

\bibitem[Roe00]{Roe00}
Pascale Roesch.
\newblock Holomorphic motions and puzzles (following {M}. {S}hishikura).
\newblock In {\em The {M}andelbrot set, theme and variations}, volume 274 of
  {\em London Math. Soc. Lecture Note Ser.}, pages 117--131. Cambridge Univ.
  Press, Cambridge, 2000.

\bibitem[Rue76]{Rue76}
David Ruelle.
\newblock A measure associated with axiom-{A} attractors.
\newblock {\em Amer. J. Math.}, 98(3):619--654, 1976.

\bibitem[Sar11]{Sar11}
Omri~M. Sarig.
\newblock Bernoulli equilibrium states for surface diffeomorphisms.
\newblock {\em J. Mod. Dyn.}, 5(3):593--608, 2011.

\bibitem[Sin72]{Sin72}
Ja.~G. Sina{\u\i}.
\newblock Gibbs measures in ergodic theory.
\newblock {\em Uspehi Mat. Nauk}, 27(4(166)):21--64, 1972.

\bibitem[SU03]{StrUrb03}
Bernd~O. Stratmann and Mariusz Urbanski.
\newblock Real analyticity of topological pressure for parabolically
  semihyperbolic generalized polynomial-like maps.
\newblock {\em Indag. Math. (N.S.)}, 14(1):119--134, 2003.

\bibitem[Urb03]{Urb03c}
Mariusz Urba{\'n}ski.
\newblock Thermodynamic formalism, topological pressure, and escape rates for
  critically non-recurrent conformal dynamics.
\newblock {\em Fund. Math.}, 176(2):97--125, 2003.

\bibitem[UZ09]{UrbZdu0901}
Mariusz Urbanski and Anna Zdunik.
\newblock Ergodic theory for holomorphic endomorphisms of complex projective
  spaces.
\newblock 2009.

\bibitem[VV10]{VarVia10}
Paulo Varandas and Marcelo Viana.
\newblock Existence, uniqueness and stability of equilibrium states for
  non-uniformly expanding maps.
\newblock {\em Ann. Inst. H. Poincar\'e Anal. Non Lin\'eaire}, 27(2):555--593,
  2010.

\bibitem[You99]{You99}
Lai-Sang Young.
\newblock Recurrence times and rates of mixing.
\newblock {\em Israel J. Math.}, 110:153--188, 1999.

\bibitem[Zwe05]{Zwe05}
Roland Zweim{\"u}ller.
\newblock Invariant measures for general(ized) induced transformations.
\newblock {\em Proc. Amer. Math. Soc.}, 133(8):2283--2295 (electronic), 2005.

\end{thebibliography}

\end{document}